\documentclass[11pt]{amsart}
\usepackage{amssymb, enumerate, mdwlist, MnSymbol}

\evensidemargin\paperwidth
\advance\evensidemargin-\textwidth
\advance\oddsidemargin-1in
\evensidemargin\oddsidemargin

\topmargin\paperheight
\advance\topmargin-\textheight
\advance\topmargin-1in

\theoremstyle{plain}
\newtheorem{theorem}{Theorem}[section]
\theoremstyle{plain}
\newtheorem{proposition}[theorem]{Proposition}
\theoremstyle{plain}
\newtheorem{lemma}[theorem]{Lemma}
\theoremstyle{plain}
\newtheorem{corollary}[theorem]{Corollary}
\theoremstyle{plain}

\theoremstyle{plain}
\newtheorem{question}[theorem]{Question}
\theoremstyle{plain}
\newtheorem{mthm}{Theorem}

\theoremstyle{plain}

\theoremstyle{definition}
\newtheorem{definition}[theorem]{Definition}
\theoremstyle{remark}
\newtheorem{remark}[theorem]{Remark}
\theoremstyle{remark}

\theoremstyle{remark}

\title[Amenability versus non-exactness]
{Amenability versus non-exactness of dense subgroups of a compact group.}
\author{Masato Mimura}
\thanks{The author is supported in part by JSPS KAKENHI Grant Number JP17H04822.}
\address{Masato Mimura\\
Mathematical Institute, Tohoku University, Japan\ /\ \'{E}cole Polytechnique F\'{e}d\'{e}rale de Lausanne, Switzerland}
\email{mimura-mas@m.tohoku.ac.jp}

\date{\today}

\begin{document}
\maketitle

\begin{abstract}
Given a countable residually finite group, we construct a compact group $K$ and two elements $w$ and $u$ of $K$ with the following properties: The group generated by $w$ and $u^3$ is amenable,  the group generated by $w$ and $u$ contains a copy of the given group, and these two groups are dense in $K$. By combining it with a construction of non-exact groups that are LEF by Osajda and Arzhantseva--Osajda and formation of diagonal products, we construct an example for which the latter dense group is non-exact. Our proof employs approximations  in the space of marked groups of LEF (``Locally Embeddable into Finite groups'') groups.
\end{abstract}




\section{Introduction}
In the present paper, we provide a way to construct \textit{RF} (Residually Finite) groups with interesting properties. More precisely, we  prove certain embedding theorems into RF groups. Our main point is that we extend our framework from RF groups to \textit{LEF} (Locally Embeddable into Finite groups) groups, which are closely related to the \textit{space of marked groups}. See Section~\ref{section=Organization} for more details of our method and organization of the current paper. Here we briefly recall concepts of RF and LEF groups, which are equivalent to the standard definitions for finitely generated groups. See Subsection~\ref{subsection=TheSpaceOfMarkedGroups} for details of terminologies appearing in the definition below.

\begin{definition}\label{definition=LEF}
Let $G$ be a finitely generated group.
\begin{enumerate}[$(1)$]
  \item (\cite{BookMalcev}, \cite{Stepin}, \cite{VershikGordon}) The group $G$ is said to be \textit{LEF} if for some (equivalently every) marking $S$ of $G$, there exists a sequence of finite marked groups such that it converges to $(G;S)$ in the space of marked groups.
  \item The group $G$ is said to be \textit{RF} if moreover, we can take a convergence sequence as in $(1)$ such that it consists of marked group quotients of $(G;S)$.
\end{enumerate}
\end{definition}

One of the  motivations to construct RF groups with specified properties comes from the work \cite{MimuraSakoPartI} and \cite{MimuraSakoPartII} of Sako and the author. There we construct two metric spaces with contrasting coarse geometric properties out of a common sequence of finite groups; this construction relates to the \textit{box space construction} of metric spaces from a RF group. Our main theorem, Theorem~\ref{mtheorem=MainTheorem}, in particular serves as a source to address the following question. 
\begin{question}\label{question=LubotzkyWeiss}
For a compact $($Hausdorff$)$ infinite group $K$ and two finitely generated dense subgroups $\Lambda_1$ and $\Lambda_2$ in $K$, how different can the group properties of $\Lambda_1$ and $\Lambda_2$ be?
\end{question}
Question~\ref{question=LubotzkyWeiss} is related to RF groups. Indeed, a result of Mal'cev implies that $\Lambda_1$ and $\Lambda_2$ in Question~\ref{question=LubotzkyWeiss} must be RF: By the Peter--Weyl theorem, every compact group is residually linear.

This question is inspired by a conjecture of Lubotzky and Weiss \cite[Conjecture~5.4]{LubotzkyWeiss}, which predicted that in the setting of Question~\ref{question=LubotzkyWeiss}, it would be impossible that $\Lambda_1$ is \textit{amenable} (see \cite[Chapter~3]{bookNowakYu}) and that $\Lambda_2$ has \textit{Kazhdan's property $(\mathrm{T})$} (see \cite[6.4]{bookNowakYu}). This conjecture was resolved in the negative by Ershov and Jaikin-Zapirain \cite[Subsection~6.3]{ErshovJaikinZapirain}  for $K$ of the form $K=\prod_{n\in \mathbb{N}_{\geq 1}}\mathrm{SL}(3n,\mathbb{F}_p)$ 
 for a fixed prime $p$, where $\mathbb{F}_q$ denotes the finite field of order $q$ for a prime power $q$.  Kassabov provided a different example using a similar idea; see \cite[Subsection~6.3]{ErshovJaikinZapirain}. 

Throughout this paper, we always equip an infinite direct product $K$ of finite groups with the product topology and regard $K$ as a compact group. We use the terminology ``generation'' for algebraic generation (as a group), even inside such a $K$. For $n\in \mathbb{N}_{\geq 1}$, a group $G$ is said to be \textit{$n$-generated} if there exist $g_1,\ldots,g_n\in G$ that generate $G$.

Our main theorem, Theorem~\ref{mtheorem=MainTheorem}, provides an answer to Question~\ref{question=LubotzkyWeiss} by proving that they can be considerably different. Moreover, we may construct such examples with the \textit{minimal numbers} of generators: $\Lambda_1$ and $\Lambda_2$ are both $2$-generated; furthermore, one generator of $\Lambda_1$ is the \textit{cube} of one generator of $\lambda_2$, and the other generator of $\Lambda_1$ \textit{equals} that of $\Lambda_2$. To state Theorem~\ref{mtheorem=MainTheorem}, we use the following terminology.

\begin{definition}\label{definition=extension}
Let 
\[
1 \ \longrightarrow \ N \ \longrightarrow \ \tilde{G} \  \longrightarrow \ G \ \longrightarrow \ 1,
\]
be a  short exact sequence of countable discrete groups. Assume that $N$ satisfies a group property $\mathcal{N}$. Then we say that $\tilde{G}$ is an $\mathcal{N}$-\textit{lift} of $G$.
\end{definition}
We use the terminology \textit{lift} rather than extension because there is ambiguity of expressing extensions (``$G$-by-$N$'' or ``$N$-by-$G$''). The property $\mathcal{N}$ of our concern is being (countable) \textit{locally finite}, that means, for every non-empty subset $F\subseteq N$, the group $\langle F\rangle$ generated by $F$ is finite. Locally finite groups are amenable and have \textit{asymptotic dimension} $0$ (see \cite[2.2]{bookNowakYu}).

For two countable discrete groups $G$ and $H$, we define the \textit{standard $($restricted$)$ wreath product} $G\wr H$ by $(\bigoplus_{H}G)\rtimes H$, where $H$ acts on $\bigoplus_{H}G$ by permutations of indices by right multiplications. For $n\in \mathbb{N}_{\geq 1}$, denote by $[n]$ the set $\{1,2,\ldots ,n\}$. For a finite set $B$, define $\mathrm{Alt}(B)$ as the \textit{alternating group} over $B$. 
\begin{mthm}[Main theorem: Embedding theorem in the context of finitely generated dense subgroups in a compact group]\label{mtheorem=MainTheorem}
Let $G$ be a finitely generated LEF group. 
Then, there exist
\begin{itemize}
  \item a sequence $(l_m)_{m\in \mathbb{N}}$ of strictly increasing natural numbers at least $5$, 
  \item a sequence $(p_m)_{m\in \mathbb{N}}$ of mutually distinct primes, and 
  \item three elements $w$, $t$ and $u$ in 
\[
K=\prod_{m\in \mathbb{N}} (\mathrm{Alt}([l_m])\wr (\mathbb{Z}/p_m\mathbb{Z}))
\]
\end{itemize}
such that the following three assertions are all satisfied:
\begin{enumerate}[$(1)$]
  \item The group $\Lambda_1=\langle w,t\rangle$ is a locally-finite-lift of $\mathbb{Z}$;
  \item the group $\Lambda_2=\langle w,u\rangle$  contains an isomorphic copy of  a locally-finite-lift of $G$;
  \item these two subgroups $\Lambda_1$ and $\Lambda_2$ are both dense in $K$.
\end{enumerate}
Moreover, we may take $t=u^3$.

Further, for a given finitely generated RF group $H$, we may arrange $(l_m)_{m}$, $(p_m)_m$, $w$, $t$, and $u$  such that $\Lambda_2$ above contains an isomorphic copy of $H$.
\end{mthm}

\begin{remark}\label{remark=LFNF-lifts1}
In fact, the group $\Lambda_1$ which appears in Theorem~\ref{mtheorem=MainTheorem} may be described in more detail in terms of locally-\textit{fully}-normally-finite-lifts, abbreviated as \textit{LFNF-lifts}; then it relates to certain group properties in terms of random walks, for instance, the Liouville property and Shalom's property $H_{\mathrm{FD}}$ \cite{ShalomHFD}. See Definition~\ref{definition=LFNF-lifts} for the definition of LFNF-lifts, and see Remark~\ref{remark=LFNF-lifts2} for more precise statements on $\Lambda_1$.
\end{remark}

\begin{remark}\label{remark=lifts}
Through diagonal products (Lemma~\ref{lemma=DiagonalProducts}), the statement of Theorem~\ref{mtheorem=MainTheorem} for the LEF group $G$ follows from that for the RF group $H$. Nevertheless, we state Theorem~\ref{mtheorem=MainTheorem} as in the form above because we first prove the assertion for $G$ and then upgrade it to that for $H$; compare with Subsection~\ref{subsection=ProofOfMainTheorem}. On the assertion on $t=u^3$, see also Remark~\ref{remark=tandu}.
\end{remark}

\begin{remark}\label{remark=Wilson}
Remarkably, Wilson \cite{Wilson} showed that every countable RF  group is embedded into a $2$-generated RF group. Hence, we may drop the finite generation condition on a (countable) $H$ in the last assertion of Theorem~\ref{mtheorem=MainTheorem}.
\end{remark}

We emphasize that, even our goal is to construct RF groups, \textit{it is of importance in our construction to extend our framework to LEF groups}. For instance, if $G$ in Theorem~\ref{mtheorem=MainTheorem} is  RF, then the last statement of Theorem~\ref{mtheorem=MainTheorem} shows that we can embed $G$ into a $2$-generated RF group $\Lambda_2$ that satisfies all of the conditions of Theorem~\ref{mtheorem=MainTheorem}. Even though in this case the statement above is expressed inside the framework of RF groups, a clear way to prove it may be to consider \textit{LEF approximations} (see Subsection~\ref{subsection=TheSpaceOfMarkedGroups}), which do not come from group quotients. See Section~\ref{section=Organization} on the outline of the proof of Theorem~\ref{mtheorem=MainTheorem} and on some significance of use of LEF groups in the proof.

As a byproduct of Theorem~\ref{mtheorem=MainTheorem}, we may have $\Lambda_1$ being a locally-finite-lift of $\mathbb{Z}$  and $\Lambda_2$ being \textit{non-exact} in the setting of Question~\ref{question=LubotzkyWeiss}; see Corollary~\ref{corollary=NonExact} for the detailed statement. \textit{Exactness} for countable groups can be defined as admitting an \textit{amenable action}, in the sense of Anantharaman-Delaroche \cite[Definition~2.1]{OzawaICM}, on some compact Hausdorff space by homeomorphisms. Amenability is equivalent to saying that \textit{every} action on a compact Hausdorff space by homeomorphisms is amenable. From this point of view, non-exactness may be seen as an \textit{extreme negation} of amenability. Non-exactness of groups  has been considered as a pathological property of a group. Corollary~\ref{corollary=NonExact} follows from existence of (finitely generated) non-exact LEF groups due to Osajda \cite{OsajdaRF}, which is built upon the earlier work of Osajda \cite{Osajda} and Arzhantseva--Osajda \cite{ArzhantsevaOsajda}; see Section~\ref{section=NonExact}, specially Remark~\ref{remark=Osajda}.

One motivation to study finitely generated \textit{dense} subgroup $\Lambda$ of a compact group $K$ is that it provides a natural action by left multiplication \[
\Lambda\curvearrowright K,
\]
which is called a \textit{compact action}. By density of $\Lambda$ in $K$, this action is (set-theoretically) free and minimal. Moreover, with respect to the Haar measure on $K$, it is measure preserving and ergodic; see Remark~\ref{remark=Sawicki}. In our examples, the compact group $K$ is always \textit{profinite}; see Subsection~\ref{subsection=Profinite} for profiniteness. Hence, we, moreover, obtain \textit{profinite actions}.

For ergodic actions on a compact probability measure space, the \textit{spectral gap property}, see Remark~\ref{remark=Sawicki} for the definition, has been paid strong attention in relation to harmonic analysis,  rigidity and other fields. For instance, Gamburd, Jakobson and Sarnak \cite{GamburdJakobsonSarnak} made the \textit{spectral gap conjecture} on random dense generators of $\mathrm{SU}(2)$; see \cite{BourgainGamburd} for recent development in this conjecture.

The group $\Lambda_2$ appearing in Theorem~\ref{mtheorem=MainTheorem} \textit{never} has the spectral gap property for the corresponding profinite action $\Lambda_2\curvearrowright K$. Our second theorem enables us to obtain a new group $\Lambda_4$ \textit{with the spectral gap property} for its profinite action. (However, in that case, our conditions on the set of generators become considerably weak.) 

\begin{mthm}[Variety of profinite actions for the same underlying space]\label{mtheorem=SpectralGap}
Let $H$ be a countable RF group. Then, there exists a compact group $K$ and a finitely generated dense subgroup $\Lambda_4$ such that all of the following hold true:
\begin{itemize}
  \item The group $\Lambda_4$ contains an isomorphic copy of $H$; 
  \item the group $\Lambda_4$ admits a finitely generated subgroup $\Lambda_1$ that is a locally-finite-lift of $\mathbb{Z}$ and dense in $K$; and
  \item the action $\Lambda_4\curvearrowright K$ has a $\mathrm{spectral}$ $\mathrm{gap}$ with respect to the Haar probability measure  of $K$.
\end{itemize}
\end{mthm}
In particular, in the same way as in Section~\ref{section=NonExact}, we can take $\Lambda_4$ above to be \textit{non-exact}.

In Theorem~\ref{mtheorem=SpectralGap}, we may take
\[
K=\prod_{m\in \mathbb{N}}\mathrm{SL}(l_m,\mathbb{F}_p),
\]
for a prime $p$ and for a certain increasing sequence $(l_m)_m$. Both of the two profinite actions $\Lambda_4\curvearrowright K$ and $\Lambda_1\curvearrowright K$ constructed from $K$, $\Lambda_1$ and $\Lambda_4$ as in Theorem~\ref{mtheorem=SpectralGap} are given  by projective systems with the \textit{common} sequence of finite groups
\[
\Lambda_4\curvearrowright K=\varprojlim_{n} (\Lambda_4\curvearrowright \prod_{m\in \mathbb{N}_{\leq n}}\mathrm{SL}(l_m,\mathbb{F}_p)),\quad \Lambda_1\curvearrowright K=\varprojlim_{n} (\Lambda_1\curvearrowright \prod_{m\in \mathbb{N}_{\leq n}}\mathrm{SL}(l_m,\mathbb{F}_p));
\]
see Subsection~\ref{subsection=SpectralGap} for the construction. Furthermore, if we consider them as $\Lambda_4\curvearrowright K\curvearrowleft \Lambda_1$, where $\Lambda_4$ acts on $K$ from the left and $\Lambda_1$ does from the right, then we have two \textit{commuting} profinite actions with contrasting behaviors.

\begin{remark}\label{remark=LW}
Posterior to this work, building upon it, the author obtains some extreme counterexamples to the Lubotzky--Weiss conjecture; see \cite[Theorems~2.1 and 4.8]{MimuraLW} for the statements. These results may be seen as strengthening of Theorem~\ref{mtheorem=SpectralGap}. However, it is \textit{un}clear whether we may arrange $\Lambda_1$ and $\Lambda_2$  (playing a role similar to $\Lambda_4$ in Theorem~\ref{mtheorem=SpectralGap}) in these results in such a way that $\Lambda_1\leqslant \Lambda_2$; for Theorem~\ref{mtheorem=SpectralGap} it is possible, as we stated in that way.
\end{remark}

In this paper, we refer the reader to comprehensive treatments on coarse geometry (and concerning group properties), exactness of countable groups, and the space of marked groups (and the LEF property), respectively, to \cite{bookNowakYu}, \cite{OzawaICM}, and \cite{VershikGordon}, \cite{MimuraSakoPartI} and \cite{MimuraSakoPartII}. The reader who is not familiar with either of these topics will find further references from these treatises above. We strongly suggest the reader consult  Lemma~\ref{lemma=WreathProducts}, Remark~\ref{remark=Gruenberg} and Lemma~\ref{lemma=Absorption} in order to have some intuition of the \textit{local} point of view in the space of marked groups, which plays a key role to our proof of Theorem~\ref{mtheorem=MainTheorem}.


\section{Strategy of the proof of Theorem~\ref{mtheorem=MainTheorem} \ and organization of this paper}\label{section=Organization}
As we highlighted in the introduction, one main novelty of the current paper is that we \textit{enlarge the framework to LEF groups}, from RF groups, to construct RF groups with specified properties. The following two points may be crucial to our construction:
\begin{itemize}
  \item The LEF property is \textit{closed} inside the space of marked groups; see Lemma~\ref{lemma=WreathProducts}. This means, for a convergent sequence of marked groups in the Cayley topology, if each marked group in the sequence is LEF, then so is the limit. See the proof of Theorem~\ref{mtheorem=MainTheorem}, more precisely Step~3 in Subsection~\ref{subsection=ProofOfMainTheorem}, how we utilize this closeness property. 

On the other hand, the RF property is \textit{not} closed; see Remark~\ref{remark=Gruenberg}. It implies that it is sometimes much easier to construct a LEF group with specified property than to obtain such a RF group; compare with Remark~\ref{remark=Osajda}.
  \item The LEF property for finitely generated groups is \textit{closed under taking standard $($restricted$)$ wreath products}. This provides us room that suffices to apply variants of \textit{Hall's embedding argument} (\cite[1.5]{Hall}) and of the \textit{absorption trick} (\cite[Lemma~6.13]{BartholdiErschler}). The former is employed to reduce the number of generators to $2$; the latter is used to construct two system of markings of a fixed sequence of finite groups that have considerably different behaviors at the Cayley limits (see Step~3 in the outlined proof of Theorem~\ref{mtheorem=MainTheorem} below). See, respectively, Lemma~\ref{lemma=Hall} and Lemma~\ref{lemma=Absorption} for precise statements; we integrate these two arguments into a key proposition; Proposition~\ref{proposition=Absorption}.

In contrast, permanence of the RF property under formation of standard wreath products is \textit{quite restrictive}; see Remark~\ref{remark=Gruenberg}.
\end{itemize}
Concerning the first point, one of the major open problems on geometric group theory asks whether all (Gromov-)hyperbolic groups are RF. If the answer to this problem is affirmative, then all groups constructed as (infinitely presented) limits of (possibly graphical) finitely presented small cancellation groups will be automatically LEF. By contrapositive, it follows that if there exists one infinitely presented small cancellation group that is not LEF, then it will resolve the problem above in the negative.

The price here to pay for switching our framework from RF groups to LEF ones is that our outcome is only a LEF group, not a RF group in general. However, the \textit{diagonal product} of marked groups enables us to have a \textit{RF} group out of a LEF group such that it is an \textit{LFNF-lift}, see Definition~\ref{definition=LFNF-lifts}, of the original LEF group (in particular, it is a locally-finite-lift); see Lemma~\ref{lemma=DiagonalProducts}. Note that Wilson \cite{Wilson} argued in a similar way to one as in the original embedding argument of Hall, which was inspired by the work of B. H. Neumann and H. Neumann \cite{NeumannNeumann}, by employing standard unrestricted wreath products; Wilson considered a split extension of an infinite products of them and recovered the RF property to establish the aforementioned result in Remark~\ref{remark=Wilson}. However, in our motivation concerning Question~\ref{question=LubotzkyWeiss}, it may not be clear whether we can take a similar strategy to that.

The proof of Theorem~\ref{mtheorem=MainTheorem} consists of the following four main steps.
\begin{enumerate}
  \item[\textit{Step~$1$}.] For a given finitely generated LEF group $G$, embed it into a \textit{LEF} group $G^{\#}$ that is generated by (finitely many) torsions; see Lemma~\ref{lemma=Auxiliary} in Subsection~\ref{subsection=Auxiliary}. This step is used to obtain the final group $\Lambda_1$ that is a locally-finite-lift of $\mathbb{Z}$.
  \item[\textit{Step~$2$}.] Encode information of a LEF approximations of the  $G^{\#}$ above into alternating groups; see Lemma~\ref{lemma=AlternatingGroups} in Subsection~\ref{subsection=Encoding}. This step is important to ensure that the final groups $\Lambda_1$ and $\Lambda_2$ are \textit{dense} in $K$, as well as to obtain an isomorphic copy of $G^{\#}$ inside one of the Cayley limit group of a sequence of $2$-marked groups. For the former, see Lemma~\ref{lemma=Goursat}, which is a byproduct of Goursat's lemma.
  \item[\textit{Step~$3$}.] Combine a variant of an \textit{embedding argument of Hall}  and that of \textit{the absorption trick}; see, respectively, Lemma~\ref{lemma=Hall} in Subsection~\ref{subsection=Hall} and Lemma~\ref{lemma=Absorption} in Subsection~\ref{subsection=AbsorptionTrick} for basic ideas of these arguments. Obtain two systems of  \textit{$2$-markings} of finite groups which are related to a LEF approximation of $G^{\#}$ that satisfies the following conditions: With respect to one marking, the Cayley limit group is $(\mathbb{Z}/2\mathbb{Z})\wr \mathbb{Z}$; with respect to the other, the Cayley limit group contains a copy of $G^{\#}$. This is the key step to the whole proof. See Proposition~\ref{proposition=Absorption} in Subsection~\ref{subsection=Absorption} for details. 
  \item[\textit{Step~$4$}.] Take the \textit{diagonal products} associated, respectively, with the two LEF approximations constructed in the third step; see Lemma~\ref{lemma=DiagonalProducts} in Subsection~\ref{subsection=DiagonalProducts}. By Lemma~\ref{lemma=Goursat}, this procedure provides us with desired $w$, $t$ and $u$ as in Theorem~\ref{mtheorem=MainTheorem}. 
\end{enumerate}

To prove the last assertion on a RF group $H$, we a priori take a specific LEF approximation coming from a projective system in order to construct an isomorphic copy of $H$ in the resulting $\Lambda_2$; see Subsection~\ref{subsection=Profinite}. In order to take $t=u^3$, we employ finite dihedral groups $D_{p_m}$ and apply encoding as in Step~2 twice. See Subsection~\ref{subsection=ProofOfMainTheorem} for details.

\begin{remark}\label{remark=(T)}
It is known that for an infinite $H$, the group $G\wr H$ \textit{never} has property $(\mathrm{T})$ unless $G$ is trivial. Hence, despite that Question~\ref{question=LubotzkyWeiss} is inspired by the Lubotzky-Weiss conjecture, Theorems~\ref{mtheorem=MainTheorem} or \ref{mtheorem=SpectralGap} does \textit{not} produce a dense group $\Lambda_2$ with property $(\mathrm{T})$. See Remark~\ref{remark=LW} and \cite{MimuraLW} for the further development concerning on that conjecture.
\end{remark}

The organization of the present paper goes as follows: In Section~\ref{section=NonExact}, we state Corollary~\ref{corollary=NonExact} and deduce it from Theorem~\ref{mtheorem=MainTheorem}. We explain some motivation of $(ii)$ of Corollary~\ref{corollary=NonExact}. In Section~\ref{section=Preliminaries}, we explain several ingredients of the proof of Theorem~\ref{mtheorem=MainTheorem}, including a brief introduction of the space of marked groups and diagonal products (Subsections~\ref{subsection=TheSpaceOfMarkedGroups} and \ref{subsection=DiagonalProducts}), a byproduct of  the Goursat lemma (Subsection~\ref{subsection=Goursat}), profinite completions (Subsection~\ref{subsection=Profinite}), a variant of Hall's embedding argument (Subsection~\ref{subsection=Hall}), and encoding into symmetric/alternating groups (Subsections~\ref{subsection=Encoding}). Section~\ref{section=ProofOfMainTheorem} is devoted to the proof of Theorem~\ref{mtheorem=MainTheorem}. In Subsections~\ref{subsection=Local} and \ref{subsection=AbsorptionTrick}, we describe some intuition of Cayley convergences in the space of marked groups and ideas based on it. We explain the first step (Lemma~\ref{lemma=Auxiliary} in Subsection~\ref{subsection=Auxiliary}) and the third step (Proposition~\ref{proposition=Absorption} in Subsection~\ref{subsection=Absorption}) in the outlined proof above. In Section~\ref{section=EmbeddingWithControl}, we prove Theorem~\ref{mtheorem=SpectralGap}.

\section{Application to ``amenability versus non-exactness''}\label{section=NonExact}
Here we deduce the following corollary to Theorem~\ref{mtheorem=MainTheorem}. 
See \cite[Chapter~3]{bookNowakYu} for the definition of elementary amenable groups. As we mentioned in Remark~\ref{remark=LFNF-lifts1}, we give the definition of \textit{LFNF-lifts}.

\begin{definition}[LFNF-lifts]\label{definition=LFNF-lifts}
Let $\Lambda$ and $\Gamma$ be countable groups.
\begin{enumerate}[$(1)$]
\item For a subgroup $N$ of a countable group $\Lambda$, we say that $N$ is \textit{locally fully normally finite}  if for every (non-empty) finite subset $F$ of $N$, the normal closure $\llangle F\rrangle_{\Lambda}$ of $F$ \textit{in} $\Lambda$ is finite.
\item We say $\Lambda$ is a \textit{locally-fully-normally-finite-lift}, an \textit{LFNF-lift} for short, of $\Gamma$ if it admits a short exact sequence
\[
1\quad \longrightarrow\quad N \quad \longrightarrow \quad  \Lambda \quad \longrightarrow \quad \Gamma \quad \longrightarrow \quad 1,
\]
where $N$ is locally fully normally finite.
\end{enumerate}
\end{definition}
For instance, the \textit{standard} (\textit{restricted}) \textit{wreath product} $(\mathbb{Z}/2\mathbb{Z})\wr \mathbb{Z}$, see Subsection~\ref{subsection=Local} for the definition,  is a locally-finite-lift of $\mathbb{Z}$ (it is moreover a local-normal-finite-lift of $\mathbb{Z}$ in the sense of \cite{BrieusselZhengHFD}). However, it is \textit{not} an LFNF-lift of $\mathbb{Z}$.

Note that local-full-normal-finiteness implies local finiteness. Unlike local finiteness (or local-normal-finiteness in the sense of \cite{BrieusselZhengHFD}), this property is \textit{not} ``intrinsic'', namely, this is a property as a \textit{subgroup} $N$ \textit{of} $\Lambda$, but of a group $N$ alone.

\begin{corollary}\label{corollary=NonExact}
\begin{enumerate}[$(i)$]
\item Let $\mathcal{P}$ be a group property of countable discrete groups. Assume that $\mathcal{P}$ satisfies the following two conditions:
\begin{itemize}
    \item $\mathcal{P}$ is closed under taking LFNF-lifts;
  \item $\mathcal{P}$ is closed under taking overgroups, namely, if $G\geqslant H$ and $H$ has $\mathcal{P}$, then so does $G$.
\end{itemize}
 Assume that there exists a $($finitely generated$)$ LEF group with $\mathcal{P}$. Then there exist a compact group $K$ and two elements $w,u\in K$ such that the following three assertions are fulfilled:
\begin{enumerate}[$(1)$]
  \item The group $\Lambda_1=\langle w,u^3\rangle$ is a locally-finite-lift of $\mathbb{Z}$. Furthermore, $\Lambda_1$ an LFNF-lift of $C\wr \mathbb{Z}$, where $C$ is a finite cyclic group $($may be taken as $\mathbb{Z}/2\mathbb{Z}$$)$;
  \item the group $\Lambda_2=\langle w,u\rangle$ has $\mathcal{P}$;
  \item both of $\Lambda_1$ and $\Lambda_2$ are dense in $K$.
\end{enumerate}

In the statements above, we may drop the first condition on the property $\mathcal{P}$ if there exists a $($countable$)$ RF group with $\mathcal{P}$.
\item $($Amenability versus non-exactness$)$ There exist a compact group $K$ and  $w,u\in K$ such that $\Lambda_1=\langle w,u^3\rangle$ is an LFNF-lift of $(\mathbb{Z}/2\mathbb{Z})\wr \mathbb{Z}$, $\Lambda_2=\langle w,u\rangle$ is $\mathrm{not}$-exact, and both of $\Lambda_1$ and $\Lambda_2$ are dense in $K$.
\end{enumerate}
Moreover, for a given $($countable$)$ RF group $H$, we may arrange $K$, $w$ and $u$ above $($both in $(i)$ and $(ii)$$)$ such that $\Lambda_2$ contains an isomorphic copy of $H$.
\end{corollary}
The group $\Lambda_1$ as in $(i)$ of Corollary~\ref{corollary=NonExact}, in particular, has asymptotic dimension $1$; see Remark~\ref{remark=LFNF-lifts2} for further properties of $\Lambda_1$ in relation to random walks.
See also Remark~\ref{remark=WarpedCones} for a possible application of $(ii)$ in terms of \textit{warped cones}.

Exactness (for countable groups), as we mentioned in the introduction, is known to be equivalent to \textit{property A} of Yu (see \cite[Chapter~4]{bookNowakYu}); see \cite[Theorem~2.5]{OzawaICM}. See \cite{Gromovrandomwalk} and \cite{ArzhantsevaDelzant}, as well as \cite{ArzhantsevaOsajda}, \cite{Osajda} and \cite{OsajdaRF}, for the history of constructions of non-exact groups.

\begin{remark}\label{remark=Osajda}
Item $(ii)$ of Corollary~\ref{corollary=NonExact} in particular provides an example of an  RF non-exact group $\Lambda_2$, which answers \cite[Problem 10.4.6]{bookBrownOzawa}, that is different from the first example by Osajda \cite{OsajdaRF}. In his work, the main difficulty was to ensure the RF property of the resulting group. The LEF property of it is automatic because this group is constructed as a limit in the Cayley topology of RF groups; \textit{such a group is LEF because it is, in particular, a Cayley limit of LEF groups and the LEF property is a closed property.} Recall our discussion in Section~\ref{section=Organization}. In \cite{Osajda} and \cite{ArzhantsevaOsajda}, discussion on the LEF property was not explicitly written. In the later work of Osajda \cite{OsajdaRF}, the construction that satisfies the condition above was given.

In our construction, on the other hand, the RF property follows from general theory \textit{out of the LEF property}; compare with the statement of Theorem~\ref{mtheorem=MainTheorem}. However, in contrast to the groups constructed in \cite{Osajda} and \cite{ArzhantsevaOsajda}, and \cite{OsajdaRF}, the group $\Lambda_2$ in our construction may not be an (infinitely presented) graphical small cancellation group. We do not know whether our $\Lambda_2$ as in $(ii)$ of Corollary~\ref{corollary=NonExact} is \textit{a-$\mathrm{T}$-menable} (see \cite[6.2]{bookNowakYu}), whereas constructions in \cite{Osajda} and \cite{ArzhantsevaOsajda}, and \cite{OsajdaRF} provided a-$\mathrm{T}$-menable non-exact groups (see \cite[Remark~9.5]{MimuraSakoPartII}).
\end{remark}

\begin{proof}[Proof of ``Theorem~$\ref{mtheorem=MainTheorem}$ implies Corollary~$\ref{corollary=NonExact}$'']

Item $(i)$ is a direct corollary; see also Remarks~\ref{remark=Wilson} and \ref{remark=LFNF-lifts2}. To prove $(ii)$, observe that the work of \cite{OsajdaRF}, \cite{Osajda} and \cite{ArzhantsevaOsajda} implies existence of  a finitely generated LEF non-exact group; see Remark~\ref{remark=Osajda}. The other key is the following.

\begin{proposition}\label{proposition=Nowak}
Let $\mathcal{P}$ be the $\mathrm{failure}$ of exactness $($for countable discrete groups$)$. Then $\mathcal{P}$ is closed under taking overgroups. It is also closed under formation of amenable-lifts; in particular, it is closed under taking LFNF-lifts.
\end{proposition}

\begin{proof}
It is well known that exactness is closed under taking subgroups (see discussion below Definition~2.4  in \cite{OzawaICM}). Nowak \cite[Theorem~2]{Nowak} showed that exactness is closed under taking quotients by amenable normal subgroups. (Nowak proved this closeness in the context of property A; it may be also proved in terms of exactness of reduced $C^{\ast}$-algebras by virtue of Theorems of Kirchberg and other researchers.) 
\end{proof}

These two keys above enable us to deduce $(ii)$ from $(i)$.
\end{proof}
We may also prove $(ii)$ by embedding  the (finitely generated) RF non-exact group, constructed by \cite{OsajdaRF}, into $\Lambda_2$.

\begin{remark}\label{remark=Sawicki}
As we mentioned in the introduction, for a finitely generated dense subgroup $\Lambda$ of a compact group $K$, we may construct the corresponding \textit{compact action} $\Lambda\curvearrowright K$; it is (set-theoretically) free, minimal (each orbit is dense), Haar-measure preserving, and Haar-\textit{ergodic}, that means, 
\[
L_2(K,\mathrm{Haar})^{\Lambda}=\mathbb{C}\mathbf{1},
\]
where $\mathbf{1}$ means the constant $1$ function on $K$. Here $\Lambda$ acts as Koopman representations: $\lambda\cdot f(\omega)=f(\lambda^{-1} \omega)$ for $\lambda\in \Lambda$ and Haar-almost all $\omega\in K$, and $L_2(K,\mathrm{Haar})^{\Lambda}$ denotes the subspace of $\Lambda$-invariant vectors. Indeed, to see ergodicity, note that $K\curvearrowright K$ is ergodic. By density of $\Lambda$ in $K$, we obtain the desired ergodicity.

Let $G$ be a finitely generated group and $G\curvearrowright (\Omega,\kappa)$ be a measure preserving action on a measure space. We say that this action, \textit{possibly non-ergodic}, has a \textit{spectral gap} if for some (equivalently, every) finite generating set $S$ of $G$, there exists $\epsilon=\epsilon_{S}>0$ such that for every $f\in L_2(\Omega,\kappa)$, it holds that
\[
\sup_{s\in S}\|\overline{s\cdot f}-\overline{f}\|_{L_2(\Omega,\kappa)/(L_2(\Omega,\kappa)^{G})}\geq \epsilon\|\overline{f}\|_{L_2(\Omega,\kappa)/(L_2(\Omega,\kappa)^{G})}.
\]
Here  $f\mapsto \overline{f}$ is the quotient map $L_2(\Omega,\kappa)\twoheadrightarrow L_2(\Omega,\kappa)/(L_2(\Omega,\kappa)^{G})$. Though the exact value of  $\epsilon_S$ be affected by the choice of $S$, strict positivity of (best possible) $\epsilon_S$ does not depend on it. Kazhdan's property $(\mathrm{T})$ for $G$ implies that every measure preserving action $G\curvearrowright (\Omega,\kappa)$ has a spectral gap, in the sense above. 

\end{remark}

\begin{remark}\label{remark=WarpedCones}
Profinite actions might have the potential to serve as a source to construct metric spaces of interest via \textit{warped cones}. Sawicki \cite[4.5]{SawickiCounterexample} constructed a class of metric spaces with some pathological property (such as the failure of the coarse Baum--Connes conjecture) such that they are \textit{not} coarsely equivalent to any sequence of graphs (in particular, they are far from being geodesic spaces in any sense). They are constructed as (level sets of) \textit{warped cones} of certain group actions $\Gamma\curvearrowright Z$ on the Cantor set $Z$. Nowak--Sawicki \cite{NowakSawicki} showed that if the action has a certain spectral gap property (with respect to some probability measure invariant under the group action) relative to a Banach space $E$, then the resulting spaces do not admit \textit{coarse embeddings} (see \cite[Chapter~5]{bookNowakYu}) into $E$. Sawicki \cite[Proposition~7.4]{SawickiPropertyA} also introduced a notion of \textit{piecewise property $A$}, and showed that under certain conditions, a warped cone has this property if and only if the group $\Gamma$ is exact.

Note that, in the statement of Corollary~\ref{corollary=NonExact}, we may take $H=\mathrm{SL}(n,\mathbb{F}_p[X])$ for $n\geq 3$, where $\mathbb{F}_p[X]$ denotes the one-variable polynomial ring with indeterminate  $X$ over $\mathbb{F}_p$ and $p$ prime. By a celebrated result of V. Lafforgue \cite{Lafforgue1}, every measure-preserving action of that group on a probability measure space has a spectral gap, in the sense of Remark~\ref{remark=Sawicki}, relative to every Banach space $E$ of \textit{non-trivial type}; see \cite[Example~4.11.$(4)$]{MimuraSakoPartII} for the definition and further references. 
\end{remark}

\section{Ingredients of the proof of Theorem~\ref{mtheorem=MainTheorem}}\label{section=Preliminaries}
In this paper,  hereafter, for each $n\in\mathbb{N}$ we set $\mathbb{M}_n=(\mathbb{N}_{\leq n}=)\{0,1,\ldots,n\}$. Our convention of the group commutator is $[\gamma_1,\gamma_2]=\gamma_1^{-1}\gamma_2^{-1}\gamma_1\gamma_2$.

\subsection{The space of marked groups}\label{subsection=TheSpaceOfMarkedGroups}
Fix $k\in \mathbb{N}_{\geq 1}$. \textit{The space of $k$-marked groups}, we write as $\mathcal{G}(k)$, was intensively studied by Grigorchuk \cite{Grigorchuk}. Here we briefly recall the definition of it and diagonal products associated to LEF approximations. We refer the reader to \cite[Subsections~2.1, 2.3 and 5.1]{MimuraSakoPartI} for more details and references. Throughout this paper, for a group $G$, write its unit as $e_G$.

A \textit{$k$-marked group} is a pair $(G;S)=(G;s_1,\ldots ,s_k)$ of a group $G$ and an ordered $k$-tuple $S=(s_1,\ldots ,s_k)$ (called a \textit{$k$-marking} of $G$) that generates $G$ (as a group). We use the symbol $\mathbf{G}$ to express a marked group. We identify two marked groups $\mathbf{G}_1=(G_1;s_1^{(1)},\ldots ,s_k^{(1)})$ and $\mathbf{G}_2=(G_2;s_1^{(2)},\ldots ,s_k^{(2)})$ if there exists a marked group isomorphism $\phi$, namely, an isomorphism $\phi\colon G_1\stackrel{\simeq}{\to}G_2$ that sends each $s_j^{(1)}$ to $s_j^{(2)}$ for $j\in [k]$. In the setting above, we say a map $\psi\colon \mathbf{G}_1\to \mathbf{G}_2$ is a \textit{marked group quotient} if $\psi\colon G_1\twoheadrightarrow G_2$ is a homomorphism that sends each $s_j^{(1)}$ to $s_j^{(2)}$ for $j\in [k]$. Each $k$-marked group $\mathbf{G}=(G;s_1,\ldots ,s_k)$ corresponds to a combinatorial object, the (right) \textit{Cayley diagram} $\mathrm{CayD}(\mathbf{G})$. It is a graph with edge colorings with color set $[k]$ and with edge orientations, defined as follows. The vertex set is $G$, and for each $j\in [k]$, we draw an edge with orientation from $g$ to $s_jg$ in color $j(\in [k])$. We endow it with the shortest-path distance $d_{\mathbf{G}}$ (here we ignore edge orientations to consider shortest paths). For each $R\in \mathbb{N}$ and $g\in G$, denote by $B_{\mathrm{CayD}(\mathbf{G})}(g,R)$ the (closed) $R$-ball centered at $g$. It is equipped with the structure of a \textit{rooted diagram}, more precisely, we set the root as $g\in G$, the vertex set as $B_{\mathbf{G}}(g,R)=\{h\in G:d_{\mathbf{G}}(h,g)\leq R\}$, and the edge set as the set of all edges whose both terminal points belong to $B_{\mathbf{G}}(g,R)$, with remembering all of those edge-colorings and edge-orientations.

The space $\mathcal{G}(k)$ is endowed with a natural topology, which is \textit{compact} and metrizable. It is called the \textit{Cayley topology} in some literature. The convergence  in the Cayley topology corresponds to the \textit{local convergence of rooted diagrams}. More precisely, $\mathbf{G}_m \to \mathbf{G}_{\infty}$ in the Cayley topology if and only if the following holds: For every $R\in \mathbb{N}$, there exists $m_R\in \mathbb{N}$ such that for every $m\in \mathbb{N}_{\geq m_R}$, the two rooted diagrams $B_{\mathrm{CayD}(\mathbf{G_m})}(e_{G_m},R)$ and $B_{\mathrm{CayD}(\mathbf{G_{\infty}})}(e_{G_{\infty}},R)$ are isomorphic (\textit{as rooted diagrams}). Here two rooted diagrams $\mathcal{D}_1$ and $\mathcal{D}_2$, with the common edge color set $[k]$, is \textit{isomorphic} if there exists a graph automorphism between them that sends the root of $\mathcal{D}_1$ to the root of $\mathcal{D}_2$ and that preserves the edge color (in $[k]$) and the edge orientation of each edge. We write the convergence $\mathbf{G}_m \to \mathbf{G}_{\infty}$ in the Cayley topology as $\mathbf{G}_m \stackrel{\mathrm{Cay}}{\to} \mathbf{G}_{\infty}$. 

In more group theoretic language, the Cayley convergence is described as follow: $\mathbf{G}_m=(G_m;s_1^{(m)},\ldots ,s_k^{(m)})$ converges to $\mathbf{G}_{\infty}=(G_{\infty};s_1^{(\infty)},\ldots ,s_k^{(\infty)})$ in the Cayley topology if and only if the map sending $s_j^{(m)}$ to $s_j^{(\infty)}$ is a \textit{partial isomorphism} from the $B_{\mathbf{G}_m}(e_{G_m},R)$ to $B_{\mathbf{G}_{\infty}}(e_{G_{\infty}},R)$. Here, for non-empty subsets $A\subseteq \Gamma_1$ and $B\subseteq \Gamma_2$, a map $\phi\colon A\to B$ is called a \textit{partial homomorphism} if for every $\gamma,\gamma'\in A$ with $\gamma\gamma'\in A$, 
\[
\phi(\gamma\gamma')=\phi(\gamma)\phi(\gamma')
\]
holds true. A \textit{partial isomorphism} is defined to be a bijective partial homomorphism. 

If a convergent sequence $\mathbf{G}_m \stackrel{\mathrm{Cay}}{\to} \mathbf{G}_{\infty}$ is such that for every $m\in \mathbb{N}$, $\mathbf{G}_m$ is a finite marked group (namely, the underlying group $G_m$ of $\mathbf{G}_m$ is finite), then we call the sequence $(\mathbf{G}_m)_{m\in \mathbb{N}}$ a \textit{LEF approximation} of $\mathbf{G}_{\infty}$. 

Note that the RF and the LEF properties are, respectively, closed under taking finitely generated subgroups (these properties may be defined for countable groups and this closeness holds in that context, but in this paper we only need it among finitely generated subgroups). In fact, for the LEF property for finitely generated groups, the following is easily seen from the characterization above of Cayley convergence. In this paper, for $\ell\in \mathbb{N}_{\geq 1}$Cwe denote by $F_\ell$ the (non-abelian) free group of rank $\ell$.

\begin{lemma}\label{lemma=LEFsubgroups}
Let $((L_m;v_1^{(m)},\ldots ,v_{\ell}^{(m)}))_{m\in \mathbb{N}}$ be a convergent sequence to $(L_{\infty};v_1^{(\infty)},\ldots ,v_{\ell}^{(\infty)})$ in the Cayley topology. Let $k\in \mathbb{N}_{\geq 1}$ and let $\omega_1,\ldots, \omega_k$ be words in $F_{\ell}$. For every $m\in \mathbb{N}\cup\{\infty\}$, let $s_j^{(m)}=\omega_j(v_1^{(m)},\ldots ,v_{\ell}^{(m)})$ for every $j\in [k]$. Let $G_m$ be the subgroup of $L_m$ generated by $s_j^{(m)}$, $j\in [k]$. Then $((G_m;s_1^{(m)},\ldots ,s_k^{(m)}))_{m\in \mathbb{N}}$ converges to $(G_{\infty};s_1^{(\infty)},\ldots ,s_k^{(\infty)})$ in the Cayley topology.
\end{lemma}

In particular, if we a priori know that $((G_m;s_1^{(m)},\ldots ,s_k^{(m)}))_{m\in \mathbb{N}}$ converges to a marked group $(G;s_1,\ldots ,s_k)$ in the Cayley topology in the setting above, then $L_{\infty}$ contains an isomorphic copy of $G$. Throughout this paper, we will freely use Lemma~\ref{lemma=LEFsubgroups} without mentioning that.

\subsection{Diagonal products}\label{subsection=DiagonalProducts}
For a sequence $(\mathbf{L}_m)_{m\in \mathbb{N}}=((L_m; v_1^{(m)},\ldots ,v_{\ell}^{(m)}))_{m}$ in $\mathcal{G}({\ell})$, we can construct an ${\ell}$-marked group by formation of the \textit{diagonal product} as follows; see also Brieussel and Zheng \cite[Subsection~2.1]{BrieusselZheng}. It was called the $\bigotimes$-product in the paper of Kassabov and Pak \cite[Definition~4.1]{KassabovPak}. 

We define the \textit{diagonal product} of a (sub)sequence of $\ell$-marked groups as follows: Let $\mathbb{M}=\{m_1,m_2,\ldots\}$ is a non-empty (possibly finite) subset of $\mathbb{N}$. Let $K^{(\mathbb{M})}=\prod_{m\in \mathbb{M}}L_m$ and set for each $j\in [{\ell}]$,
\[
v_j^{(\mathbb{M})}=(v_j^{(m_1)},v_j^{(m_2)},\ldots ,v_j^{(m_i)},\ldots ,) \quad (\in K^{(\mathbb{M})}).
\]
Then, $\Delta_{m\in \mathbb{M}}(\mathbf{L}_m)$ is defined as the marked group
\[
\Delta_{m\in \mathbb{M}}(\mathbf{L}_m)=(\Lambda^{(\mathbb{M})};v_1^{(\mathbb{M})},v_2^{(\mathbb{M})},\ldots,v_{\ell}^{(\mathbb{M})}).
\]
Here $\Lambda^{(\mathbb{M})}$ is the group  generated by $(v_1^{(\mathbb{M})},\ldots,v_{\ell}^{(\mathbb{M})})$. If we take $\mathbb{M}=\mathbb{N}$, then in this paper, we simply write $K^{(\mathbb{M})}$, $\Lambda^{(\mathbb{M})}$ and  $v_1^{(\mathbb{M})},\ldots,v_{\ell}^{(\mathbb{M})}$, respectively, as $K$, $\Lambda$ and $v_1, \ldots, v_{\ell}$ for short.

Now suppose that $(\mathbf{L}_m)_{m\in \mathbb{N}}=((L_m; v_1^{(m)},\ldots ,v_{\ell}^{(m)}))_{m}$ is a convergent sequence to $\mathbf{L}_{\infty}=(L_{\infty}; v_1^{(\infty)},\ldots ,v_{\ell}^{(\infty)})$. Then the relationship between $\Lambda$ above and $L_{\infty}$ is explained in the following way: The direct sum $\bigoplus_{m\in \mathbb{N}}L_m \leqslant K$ is a normal subgroup of $K$. The map that sends $v_j$ to $v_j^{(\infty)}$ for each $j\in [{\ell}]$ induces the following short exact sequence
\[
1\ \longrightarrow \ \Lambda \cap \bigoplus_{m\in \mathbb{N}}L_m \ \longrightarrow \ \Lambda \ \longrightarrow \ L_{\infty}\ \longrightarrow \  1.
\]
See \cite[Lemma~4.6]{KassabovPak}. In particular, we have the following.

\begin{lemma}[Diagonal products and LFNF-lifts]; \cite{KassabovPak}\label{lemma=DiagonalProducts}
Let $(\mathbf{L}_m)_{m\in \mathbb{N}}$ be a LEF approximation of $\mathbf{L}_{\infty}=(L_{\infty}; v_1^{(\infty)},\ldots ,v_{\ell}^{(\infty)})$.  Then the following hold true:
\begin{enumerate}[$(1)$]
  \item The underlying group $\Lambda$ of $\Delta_{m\in \mathbb{N}}(\mathbf{L}_m)$ is a RF group.
  \item The group $\Lambda$ is a subgroup of $\prod_{m\in \mathbb{N}}L_m$.
  \item The group $\Lambda$ is an LFNF-lift of $L_{\infty}$; in particular, it is a locally-finite-lift of $L_{\infty}$.
\end{enumerate}
\end{lemma}

In general, $\Lambda$ as in this lemma is not identical to the Cayley limit $L_{\infty}$. See Subsection~\ref{subsection=Profinite} for a case where $L_{\infty}$ is \textit{isomorphically} lifted. 

We include the following lemma due to its own interest. We do not use it in this paper; in our construction, we can take a well-chosen LEF approximation of a finitely generated RF group $H$ in the proof of the last assertion of Theorem~\ref{mtheorem=MainTheorem}; compare with Subsection~\ref{subsection=Profinite}. This  lemma deals with an \textit{arbitrary} LEF approximation of $H$, and may be combined with Lemma~\ref{lemma=LEFsubgroups}. 

\begin{lemma}\label{lemma=FinitelyPresented}
Let $H$ be a $\mathrm{finitely}$ $\mathrm{presented}$ RF group and $(h_1^{(\infty)},\ldots,h_{\ell}^{(\infty)})$ be a marking of $H$. Then for every LEF approximation $(\mathbf{H})_{m\in \mathbb{N}}=((H_m;h_1^{(m)},\ldots,h_{\ell}^{(m)}))_{m}$ of the marked group $(H;h_1^{(\infty)},\ldots,h_{\ell}^{(\infty)})$, after deleting finitely many $\mathbf{H}_m$ from the LEF approximation if necessary, we have that $\Delta_{m\in \mathbb{N}}(\mathbf{H}_m)$ is isomorphic to $(H;h_1^{(\infty)},\ldots,h_{\ell}^{(\infty)})$ as a marked group. In particular, the underlying group $\tilde{H}$ of $\Delta_{m\in \mathbb{N}}(\mathbf{H}_m)$ is isomorphic to $H$.
\end{lemma}

\begin{proof}
Since $H$ is finitely presented, there exists $R\in \mathbb{N}$ such that all relations in terms of $(h_1^{(\infty)},\ldots,h_{\ell}^{(\infty)})$ that suffice to define $H$ is contained in $B_{(H;h_1^{(\infty)},\ldots,h_{\ell}^{(\infty)})}(e_{H};R)$. Fix such an $R$. Then by definition of convergence in the Cayley topology, there exists $m_R$ such that for every $m\in \mathbb{N}_{\geq m_R}$, the ball in the Cayley diagram $B_{\mathrm{CayD}(H_m;h_1^{(m)},\ldots,h_{\ell}^{(m)})}(e_{H_m};R)$ is isomorphic to $B_{\mathrm{CayD}(H_{\infty};h_1^{(\infty)},\ldots,h_{\ell}^{(\infty)})}(e_{H_{\infty}};R)$ as rooted diagrams. If $m_R>0$, then delete all $\mathbf{H}_m$ for $m\in \mathbb{N}_{<m_R}$. Then for $\Delta_{m}(\mathbf{H}_m)=(\tilde{H};h_1,\ldots ,h_{\ell})$, the ${\ell}$-tuple $(h_1,\ldots ,h_{\ell})$ satisfies all relations that suffice to define $H$. This means that the map $h_j^{(\infty)}\mapsto h_j$ induces a homomorphism $H\twoheadrightarrow \tilde{H}$. On the other hand, by construction of diagonal products, the map $\colon h_j\mapsto h_j^{(\infty)}$ for $j\in [{\ell}]$ extends to a $\tilde{H}\twoheadrightarrow H$; these give an isomorphism $H\simeq \tilde{H}$.
\end{proof}

The proof above, in particular, implies that for finitely presented groups, the LEF property is equivalent to the RF property. It was proved in \cite[2. Theorem]{VershikGordon}.

\subsection{A sufficient condition for density: a byproduct of the Goursat lemma}\label{subsection=Goursat}
As we argued in Subsection~\ref{subsection=DiagonalProducts}, for a given LEF approximation $(\mathbf{L}_m)_{m\in \mathbb{N}}=((L_m;V_m))_{m}$ of $\mathbf{L}_{\infty}=(L_{\infty};V_{\infty})$, we constructed  a finitely generated RF subgroup $\Lambda$ of $K=\prod_{m\in \mathbb{N}}L_m$ as the underlying group of the diagonal product. This $\Lambda$ is \textit{not} necessarily dense in $K$. However, Lemma~\ref{lemma=Goursat}, a  byproduct of Goursat's lemma, provides a sufficient condition of density, as we will see in Lemma~\ref{lemma=Goursat}. Lemma~\ref{lemma=Goursat} might be known to the experts, but we include the proof for the sake of completeness.

The following statement is the key to density; it may be seen a corollary to Goursat's lemma in group theory. Here, we do not regard the trivial group $\{e\}$ as a finite simple group.

\begin{lemma}[Corollary to Goursat's lemma]\label{lemma=GoursatLemma}
Let $H_0$ and $H_1$ be two finite groups. For $i\in \{0,1\}$, let $\pi_i\colon H_0\times H_1\twoheadrightarrow H_i$ denote the projection onto the $i$-th coordinate. Let $G\leqslant H_0\times H_1$ satisfy $\pi_0(G)=H_0$ and $\pi_1(G)=H_1$. Assume that there does $\mathrm{not}$ exist any finite simple group that appears as a group quotient of both $H_0$ and $H_1$. Then, $G(\leqslant H_0\times H_1)$ equals the whole group $H_0\times H_1$.
\end{lemma}


\begin{proof}
Let $N_0=\{h_0\in H_0:(h_0,e_{H_1})\in G\}$ and $N_1=\{h_1\in H_0:(e_{H_0},h_1)\in G\}$. By assumption on $G$, it is straightfoward to show that for each $i\in \{0,1\}$, $N_i$ is normal in $H_i$. We then claim the isomorphism
\[
H_0/N_0\simeq H_1/N_1.
\]
To prove this, consider the following group homomorphism
\[
\rho\colon G\to H_0/N_0\times H_1/N_1;\quad (g_0,g_1)\mapsto (g_0N_0,g_1N_1).
\]
Note that $\rho$ is surjective onto each coordinate. We observe that for $g=(g_0,g_1)$ and $g'=(g_0',g_1')$ in $G$, if the $0$-th coordinate of $\rho(g)$ coincides with that of $\rho(g')$, then $\rho(g)=\rho(g')$. Indeed, the assumption above tells that $g_0^{-1}g_0'\in N_0$; then 
\[
(e_{H_0},g_1^{-1}g_1)=(g_0^{-1}g_0',e_{H_1})^{-1}(g_0^{-1},g_1^{-1})(g_0',g_1')\in G
\]
and hence $g_1^{-1}g_1'\in N_1$. By switching $H_0$ and $H_1$, we conclude that $\rho$ induces the isomorphism $H_0/N_0\simeq H_1/N_1$ (this is exactly the statement of the Goursat lemma).

Now we utilize the assumption on finite simple quotients of $H_0$ and $H_1$. The only possibility for $H_0/N_0\simeq H_1/N_1$ to happen is when the common group quotient in the equality above is the trivial group. Hence $N_0=H_0$ and $N_1=H_0$; they imply that $G=H_0\times H_1$.
\end{proof}

The following lemma roughly states that if $L_m$, $m\in \mathbb{N}$, are pairwise ``\textit{coprime}'' in terms of finite simple groups (quotients), then density is automatic. Recall also that every finite group admits a composition series, and  that a set of the simple groups appearing in this series (it is called the set of \textit{composition factors}) does not depend on the way how we take a composition series; this is the Jordan--H\"{o}lder theorem.

\begin{lemma}[Sufficient condition for density]\label{lemma=Goursat}
Let $(\mathbf{L}_m)_{m\in \mathbb{N}}=((L_m; v_1^{(m)},\ldots ,v_{\ell}^{(m)}))_{m}$ be a sequence of finite marked groups. Assume that there exists $\mathrm{no}$ finite simple group that appears as a simple group quotient of $L_m$ for two distinct $m$. Then the underlying group $\Lambda$ of $\Delta_{m\in \mathbb{N}}(\mathbf{L}_m)$ above is dense in $K=\prod_{m\in \mathbb{N}}L_m$.

In particular, if $\mathrm{no}$ finite simple group appears as a composition factor of $L_m$ for two distinct $m$, then $\Lambda$ above is dense in $K$ above.
\end{lemma}

\begin{proof}
Desired density is equivalent to saying that for every $n\in \mathbb{N}$, the projection onto coordinates from $0$-th to $n$-th, $\pi^{(\mathbb{M}_n)}\colon K\twoheadrightarrow \prod_{m\in \mathbb{M}_{n}}L_m$, gives a surjection from $\Lambda$. This equivalent form immediately follows from Lemma~\ref{lemma=GoursatLemma} by induction on $n$. 
\end{proof}

The latter (weaker) statement of Lemma~\ref{lemma=Goursat} is almost straightforward by the Jordan--H\"{o}lder theorem; see the argument in \cite[Example~4.1]{LubotzkyWeiss}.

\subsection{Relation to profinite completions}\label{subsection=Profinite}
We explain the relation between diagonal products of LEF approximations and profinite completions of RF groups. 

Let $L_{\infty}$ be a finitely generated RF group. The RF property is equivalent to existence of a \textit{nested} sequence $(N_m)_{m\in \mathbb{N}}$  (that means, for every $m\in \mathbb{N}$, it holds that $N_{m+1}\leqslant N_m$) of finite index normal subgroups of $L_{\infty}$ such that $\bigcap_{m\in \mathbb{N}}N_m=\{e_{L_{\infty}}\}$. In this paper, we call such a sequence a \textit{chain} of normal subgroups of $L_{\infty}$. For such a chain $(N_m)_{m\in \mathbb{N}}$, we can construct a projective system 
\[
L_{\infty}\twoheadrightarrow \cdots \twoheadrightarrow L_{\infty}/N_{m+1} \twoheadrightarrow L_{\infty}/N_{m} \twoheadrightarrow  \cdots \twoheadrightarrow L_{\infty}/N_{0},
\]
and the \textit{profinite completion} $\widehat{L_{\infty}}_{(N_m)_{m}}=\varprojlim_{m} (L_{\infty}/N_m)$ of $L_{\infty}$ with respect to the chain $(N_m)_{m\in \mathbb{N}}$. Let $(v_1^{(\infty)},\ldots ,v_{\ell}^{(\infty)})$ be a marking of $L_{\infty}$. Then, since $(N_m)_{m\in \mathbb{N}}$ is nested and $\bigcap_{m\in \mathbb{N}}N_m=\{e_{L_{\infty}}\}$, $((L_{\infty}/N_m; v_1^{(\infty)}\ \mathrm{mod}\ N_m,\ldots ,v_{\ell}^{(\infty)}\ \mathrm{mod}\ N_m))_{m\in \mathbb{N}}$ is a LEF approximation of  $\mathbf{L_{\infty}}=(L_{\infty};v_1^{(\infty)},\ldots ,v_{\ell}^{(\infty)}))$. Moreover, because $(L_{\infty}/N_{m})_{m\in \mathbb{N}}$ forms a projective system, it holds that for this chain $(N_m)_{m\in \mathbb{N}}$,
\begin{eqnarray*}
& &\Delta_{m\in \mathbb{M}_n}((L_{\infty}/N_m; v_1^{(\infty)}\ \mathrm{mod}\ N_m,\ldots ,v_{\ell}^{(\infty)}\ \mathrm{mod}\ N_m)) \quad  \\
\cong \quad& &(L_{\infty}/N_n; v_1^{(\infty)}\ \mathrm{mod}\ N_n,\ldots ,v_{\ell}^{(\infty)}\ \mathrm{mod}\ N_n),
\end{eqnarray*}
for every $n\in \mathbb{N}$. It implies that 
\[
\Delta_{n\in \mathbb{N}}((L_{\infty}/N_n; v_1^{(\infty)}\ \mathrm{mod}\ N_n,\ldots ,v_{\ell}^{(\infty)}\ \mathrm{mod}\ N_n))\  \cong \ \mathbf{L}_{\infty}.
\]

Now we go back to the setting of a LEF approximation of a general LEF group. For $(\Lambda;v_1,\ldots ,v_{\ell})=\Delta_{m\in \mathbb{N}}(\mathbf{L}_m)$ for a (general) LEF approximation $(\mathbf{L}_m)_{m\in \mathbb{N}}$, we set $(\Lambda^{(\mathbb{M}_n)};v_1^{(\mathbb{M}_n)},\ldots ,v_{\ell}^{(\mathbb{M}_n)})=\Delta_{m\in \mathbb{M}_n}(\mathbf{L}_m)$ (this is again a finite marked group). Then $(\Lambda^{(\mathbb{M}_n)})_{n\in \mathbb{N}}$ forms a projective system of the RF group $\Lambda$. Therefore, we have the marked group isomorphism
\[
\Delta_{n\in \mathbb{N}}(\Lambda^{(\mathbb{M}_n)};v_1^{(\mathbb{M}_n)},\ldots ,v_{\ell}^{(\mathbb{M}_n)})) \  \cong \ 
(\Lambda;v_1,\ldots ,v_{\ell}).
\]
Moreover, the following holds true: If $(L_m)_{m\in \mathbb{N}}$ satisfies the condition of (the former statement of) Lemma~\ref{lemma=Goursat}, then the profinite completion $\widehat{\Lambda}_{(N_n)_{n}}(=\varprojlim_{n}\Lambda^{(\mathbb{M}_n)})$ with respect to the chain $(N_n)_{n\in \mathbb{N}}=(\mathrm{Ker}(\Lambda\twoheadrightarrow \Lambda^{(\mathbb{M}_n)}))_{n}$ is naturally isomorphic to $\prod_{m\in \mathbb{N}}L_m$. Indeed, in this case, $\Lambda^{(\mathbb{M}_n)}=\prod_{m\in \mathbb{M}_n}L_m$ by Lemma~\ref{lemma=Goursat}. Hence, in this case, the action $\Lambda \curvearrowright K=\prod_{m\in \mathbb{N}}L_m$ may be regarded as the profinite action
\[
\varprojlim_{n}(\Lambda\curvearrowright \Lambda^{(\mathbb{M}_n)})\quad (=\varprojlim_{n}(\Lambda\curvearrowright \prod_{m\in \mathbb{M}_n}L_m));
\]
see \cite{AbertElek} for more details on profinite actions.

\subsection{Standard (restricted) wreath products and the \textit{local} point of view}\label{subsection=Local}
For two (discrete) groups $G$ and $H$, the \textit{standard}  (\textit{restricted}) \textit{wreath product} $G\wr H$ is defined as $(\bigoplus_H G)\rtimes H$, where the $H$-action is given by the permutation of coordinates from the right. We identify an element in $\bigoplus_H G$ with a map $f\colon H\to G$ such that for all but finite $h\in H$, $f(h)=e_G$. In this way, we write an element in $G\wr H$ as $(f,h)$, where $f\in \bigoplus_H G$ and $h\in H$. We write the map $f\colon H\to G$ that sends all $h$ to $e_G$ as $\mathbf{e}$. This is the group unit of $\bigoplus_H G$. For $g\in G$ and $h\in H$, we denote by $g\delta_{h}$ the element 
\[
g\delta_{h}(\gamma)=\left\{\begin{array}{cl}
g, & \textrm{if $\gamma=h$},\\
e_G, & \textrm{otherwise}
\end{array}\right.
\]
in $\bigoplus_{H} G$.

As we mentioned in Section~\ref{section=Organization}, the LEF property for finitely generated groups is closed under taking standard (restricted) wreath products. More precisely, we have the following lemma. 

\begin{lemma}[\cite{VershikGordon}]\label{lemma=WreathProducts}
Let $\mathbf{G}_m\stackrel{\mathrm{Cay}}{\to} \mathbf{G}_{\infty}$ in $\mathcal{G}(k)$ and $\mathbf{H}_n\stackrel{\mathrm{Cay}}{\to} \mathbf{H}_{\infty}$ in $\mathcal{G}(l)$. Then as $\min\{m,n\}\to \infty$, we have the following convergence in $\mathcal{G}(k+l)$:
\[
\mathbf{G}_m\wr \mathbf{H}_{n}\ \stackrel{\mathrm{Cay}}{\longrightarrow} \ \mathbf{G}_{\infty}\wr \mathbf{H}_{\infty}.
\]
Here for two marked groups $\mathbf{G}=(G;s_1,\ldots ,s_k)$ and $\mathbf{H}=(H;t_1,\ldots ,t_{\ell})$, $\mathbf{G}\wr \mathbf{H}$ is defined to be a marked group with underlying group $G\wr H$ whose marking is given by
\[
((s_1\delta_{e_H},e_H),(s_2\delta_{e_H},e_H),\ldots ,(s_k\delta_{e_H},e_H),(\mathbf{e},t_1),(\mathbf{e},t_2),\ldots ,(\mathbf{e},t_{\ell})).
\]
\end{lemma}

Although this lemma is a special case of \cite[${\S}$2.4. Theorem.$ii)$]{VershikGordon}, we include the proof of the lemma above because the proof describes some intuition of convergences in the Cayley topology; we suggest the reader consult it together with Remark~\ref{remark=Gruenberg}. As we mentioned in Subsection~\ref{subsection=TheSpaceOfMarkedGroups}, the Cayley convergence $\mathbf{G}_m\stackrel{\mathrm{Cay}}{\to} \mathbf{G}_{\infty}$ in $\mathcal{G}(k)$ is the \textit{local} convergence of rooted diagrams; more precisely, for each $R\in \mathbb{R}_{\geq 0}$, there exists $m_R\in \mathbb{N}$ such that for every $m\in \mathbb{N}_{\geq m_R}$, the $R$-ball centered at $e_{G_m}$ in the Cayley diagram $\mathrm{CayD}(\mathbf{G}_m)$ are isomorphic to the $R$-ball centered at $e_{G_{\infty}}$ in $\mathrm{CayD}(\mathbf{G}_{\infty})$. This amounts to the existence of a certain \textit{partial isomorphism} ; recall Subsection~\ref{subsection=TheSpaceOfMarkedGroups}. In this convergence, we \textit{disregard} behaviors of $\mathbf{G}_m$ \textit{outside the $R$-ball}, where $m\geq m_R$; this description will get clearer after the reader consults Remark~\ref{remark=Gruenberg}. We refer the reader to \cite[Lemma~5.1]{MimuraSakoPartI} and the proof of it as another pedagogical example.

\begin{proof}[Proof of Lemma~$\ref{lemma=WreathProducts}$]
Let $R\in \mathbb{R}_{\geq 0}$. We take $m_{R} \in \mathbb{N}$ and $n_{R} \in \mathbb{N}$, respectively, to be integers corresponding, respectively, to the Cayley convergences $\mathbf{G}_m\stackrel{\mathrm{Cay}}{\to} \mathbf{G}_{\infty}$ and $\mathbf{H}_n\stackrel{\mathrm{Cay}}{\to} \mathbf{H}_{\infty}$. In what follows, that for $m\geq m_R$ and $n\geq n_R$, the $R$-ball centered at $e_{G_m\wr H_n}$ of $\mathrm{CayD}(\mathbf{G}_m\wr \mathbf{H}_n)$ is isomorphic to the $R$-ball centered at $e_{G_{\infty}\wr H_{\infty}}$ of $\mathrm{CayD}(\mathbf{G}_{\infty}\wr \mathbf{H}_{\infty})$; once we show this, it is by definition to obtain the desired Cayley convergence.

Fix $m\in \mathbb{N}_{\geq m_{R}}$ and $n\in \mathbb{N}_{\geq n_{R}}$. By assumption, the map sending $s_j^{(m)}$ to $s_j^{(\infty)}$ for each $j\in [k]$ extends to a partial isomorphism between $B_{\mathbf{G}_m}(e_{G_m},R)$ and $B_{\mathbf{G}_{\infty}}(e_{G_{\infty}},R)$; similarly, the map sending $t_i^{(m)}$ to $t_i^{(\infty)}$ for each $i\in [l]$ extends to a partial isomorphism between $B_{\mathbf{H}_n}(e_{H_n},R)$ and $B_{\mathbf{H}_{\infty}}(e_{H_{\infty}},R)$. The key to the proof is that the $R$-ball $B_{\mathbf{G}_{m}\wr \mathbf{H}_{n}}(e_{G_{m}\wr H_{n}},R)$ is included in the following subset of $G_m\wr H_n$:
\[
\left\{(f,h):f\in \bigoplus_{B_{\mathbf{H}_{n}}(e_{H_{n}},R)} B_{\mathbf{G}_{m}}(e_{G_{m}},R),\ h\in B_{\mathbf{H}_{n}}(e_{H_{n}},R)\right\};
\]
the corresponding assertion also holds for the case that $(m,n)=(\infty,\infty)$.
Hence, via the two partial isomorphisms above, it may be seen that the map which sends $(s_j^{(m)}\delta_{e_{H_n}},e_{H_n})$ to $(s_j^{(\infty)}\delta_{e_{H_{\infty}}},e_{H_{\infty}})$ for each $j\in [k]$ and which sends $(\mathbf{e},t_i^{(m)})$ to $(\mathbf{e},t_i^{(\infty)})$ for each $i\in [l]$ extends to a paritial isomorphism from $B_{\mathbf{G}_{m}\wr \mathbf{H}_{n}}(e_{G_{m}\wr H_{n}},R)$ to $B_{\mathbf{G}_{\infty}\wr \mathbf{H}_{\infty}}(e_{G_{\infty}\wr H_{\infty}},R)$. This exactly says that the two $R$-balls in the Cayley diagrams of our concern are isomorphic. 
\end{proof}

\begin{remark}\label{remark=Gruenberg}
In this proof, it is essential that convergences in the Cayley topology only observe \textit{local behavior} of marked groups. In what follows, we explain this feature in more detail because it is the key to having some intuition of the Cayley topology. 

Consider the case where, in the proof of Lemma~\ref{lemma=WreathProducts}, $\mathbf{G}_m$ and $\mathbf{H}_n$ are, respectively, marked group quotient, respectively, of $\mathbf{G}_{\infty}$ and $\mathbf{H}_{\infty}$. In this case, the partial isomorphisms $B_{\mathbf{G}_{m}}(e_{G_{m}},R) \to B_{\mathbf{G}_{\infty}}(e_{G_{\infty}},R)$ and $B_{\mathbf{H}_{n}}(e_{H_{n}},R) \to B_{\mathbf{H}_{\infty}}(e_{H_{\infty}},R)$, in fact, come, respectively, from group quotient maps $G_{\infty}\twoheadrightarrow G_m$ and $H_{\infty}\twoheadrightarrow H_n$. More precisely, these group quotient maps are both injective inside $R$-balls, and the partial isomorphisms above are, respectively, given by the inverse maps of the restriction of these projections. We warn that, even in this case, the inverse map of the partial isomorphism $B_{\mathbf{G}_{m}\wr \mathbf{H}_{n}}(e_{G_{m}\wr H_{n}},R)\to B_{\mathbf{G}_{\infty}\wr \mathbf{H}_{\infty}}(e_{G_{\infty}\wr H_{\infty}},R)$, constructed in the proof above, may \textit{not} be extended to a group quotient map. Here, there is no problem to construct the group quotient map $G_{\infty}\wr H_{\infty}\twoheadrightarrow G_m\wr H_{\infty}$. The problem lies in \textit{the other step}, namely, when we try to construct ``the quotient map $G_{m}\wr H_{\infty}\twoheadrightarrow G_m\wr H_{n}$.''
This is because if $H_{\infty}\twoheadrightarrow H_{n}$ is a genuine projection (namely, not an isomorphism), then we need to ``fold'' the direct sum $\bigoplus_{H_{\infty}}G_{m}$ to obtain $\bigoplus_{H_{m}}G_{m}$. However, this is \textit{totally impossible unless $G_m$ is abelian}. Indeed, for every distinct pair $h_1\ne h_2\in H_{\infty}$ with the same image in $H_n$ by the quotient map $H_{\infty}\twoheadrightarrow H_{n}$, every two elements of the form $g_1\delta_{h_1}$ and $g_2\delta_{h_2}$ ($g_1,g_2\in G_m$) commute; unless $G_m$ is abelian, there is \textit{no way possible} to ``merge'' them into a single element in $G_m$ inside $\bigoplus_{H_n}G_m$. 

From this viewpoint, the proof of Lemma~\ref{lemma=WreathProducts} works because we only care \textit{local} structure of marked groups. From the \textit{local} point of view, even in the case that we consider in the paragraph above, there is \textit{no problem} to have a partial isomorphism between two standard wreath products. This is because, if we choose a sufficiently small $R$ (compared with $n$), then we have \textit{no} distinct pair $h_1\ne h_2\in H_{\infty}$ \textit{inside the $R$-ball} $B_{\mathbf{H}_{\infty}}(e_{H_{\infty}},R)$ that shares the same image by $H_{\infty}\twoheadrightarrow H_{n}$. If we restrict ourselves to this $R$-ball, then the serious  issue on ``folding'' above becomes \textit{completely harmless}.

The argument above explains why the proof of Lemma~\ref{lemma=WreathProducts} does \textit{not} imply that the standard wreath product of two finitely generated RF groups is RF. In fact, Gruenberg showed that if $G\wr H$ is RF for an infinite $H$, then $G$ must be \textit{abelian}; see \cite[Subsection~2.2]{Gruenberg}.
\end{remark}

In the present paper, we treat several convergences of marked groups in the Cayley topology; most of them are straightforward to see if the reader acquires the \textit{local} point of view in Cayley convergences. For this reason, instead of giving rigorous proofs of these convergences, we only describe intuitive ideas based on the local viewpoint, all of which will be made rigorous.

\subsection{A Hall-type argument}\label{subsection=Hall}
One of the main important points of standard (restricted) wreath products in the current paper is that it enables us to \textit{reduce the number of generators} of groups; in this subsection, we explain it. Another, even more important, point is on the \textit{absorption trick} in Cayley convergence, as we will see in Subsection~\ref{subsection=AbsorptionTrick}.

Our way of reducing the number of generators is inspired by a construction of Hall \cite[1.5]{Hall}, which has its origin in the work of B. H. Neumann and H. Neumann \cite{NeumannNeumann}. Hall employed \textit{unrestricted} wreath products in his construction (namely, instead of $G\wr H=(\bigoplus_H G)\rtimes H$, he considered $(\prod_H G)\rtimes H$); in this paper we consider a variant of it, which utilizes standard \textit{restricted} wreath products. 

\begin{lemma}[Adaptation of the Hall embedding]\label{lemma=Hall}
Let $G$ be a finitely generated LEF group and $S=(s_1^{(\infty)},\ldots, s_k^{(\infty)})$ be a $k$-marking. Consider two elements $w^{(\infty)},u^{(\infty)} \in G\wr \mathbb{Z}$, where $w^{(\infty)}=(f^{(\infty)},0)$ and $u^{(\infty)}=(\mathbf{e},1)$. Here $f^{(\infty)}=\bigoplus_{j\in [k]}s_j \delta_{2^j}$, namely,
\[
f^{(\infty)}(n)=\left\{
\begin{array}{cl}
s_j, & \textrm{if $n=2^j$ for $j\in [k]$},\\
e_G, & \textrm{otherwise.}
\end{array}
\right.
\]
Let $\Gamma=\Gamma_{(G;S)}$ be the group generated by $w^{(\infty)}$ and $u^{(\infty)}$. 

Then, $\Gamma$ is LEF and it contains an isomorphic copy of the subgroup of $G$ generated by $\{[s_i^{(\infty)},s_j^{(\infty)}]:i,j\in [k]\}$.
\end{lemma}

\begin{proof}
The LEF property of $\Gamma$ follows from Lemma~\ref{lemma=WreathProducts}. In what follows, we prove the latter assertion.

For each $i\in [k]$, let $w_i^{(\infty)}=(u^{(\infty)})^{2^i}w^{(\infty)}(u^{(\infty)})^{-2^i}=(f_i^{(\infty)},0)$. Here we set $f_i^{(\infty)}=\bigoplus_{j\in [k]}s_j \delta_{2^j-2^i}$.
The key observation here is that for $i,j,i',j'\in [k]$, the equality $2^j-2^i=2^{j'}-2^{i'}$ holds \textit{if and only if} either of the following is satisfied:
\begin{itemize}
 \item $i=j$ and $i'=j'$; in that case, the value of the equality above is $0$,
 \item $(i,j)=(i',j')$.
\end{itemize}
It implies that for every $i,j\in [k]$, $[w_i^{(\infty)},w_j^{(\infty)}]$ equals $(f_{i,j}^{(\infty)},0)$. Here we let $f_{i,j}^{(\infty)}(n)=[s_i^{(\infty)},s_j^{(\infty)}]$ for $n=0$ and $f_{i,j}^{(\infty)}(n)=e_G$ otherwise. It proves our assertion. 
\end{proof}

We employ Hall-type arguments repeatedly in the proof of Theorem~\ref{mtheorem=MainTheorem}; the significant point of this argument is that by means of standard wreath products, we create a copy of group generated by \textit{single commutators on a given set of generators} in a group with small number of generators (in the proof of Theorem~\ref{mtheorem=MainTheorem}, we will apply this to obtain $2$-generated groups and $3$-generated groups).

\subsection{Encoding into alternating/symmetric groups}\label{subsection=Encoding}

We will explain ways to encode information of a convergence sequence in the Cayley topology into alternating groups. This procedure plays a key role, in relation to Lemma~\ref{lemma=Goursat}, to \textit{show density} of finitely generated subgroups in our construction. First we discuss encoding into symmetric groups. Before proceeding to that, we collect our notation.

Let $B$ be a non-empty, at most countable set. Denote by $\mathrm{Sym}(B)$ the full symmetric group of $B$ and by $\mathrm{Sym}_{<\aleph_0}(B)$ the finitary symmetric group (the group of all permutations on $B$ with finite support). By $\mathrm{Alt}(B)$, we mean the alternating group over $B$, namely, the union of $\mathrm{Alt}(B_0)$ over all non-empty finite subsets $B_0\subseteq B$ via the natural inclusion $\mathrm{Sym}(B_0) \hookrightarrow \mathrm{Sym}(B)$.

Our encoding process into symmetric groups is explained in the following manner.
\begin{lemma}[Encoding into symmetric groups]\label{lemma=SymmetricGroups}
Let $k\in \mathbb{N}_{\geq 1}$. Let $(\mathbf{G}_{m})_{m\in \mathbb{N}}=(G_m;s_1^{(m)},\ldots ,s_k^{(m)}))_{m}$ be a LEF approximation of an infinite group $\mathbf{G}_{\infty}=(G_{\infty};s_1^{(\infty)},\ldots ,s_k^{(\infty)})$. Assume that for every $m\in \mathbb{N}\cup\{\infty\}$ and for every $j\in [k]$, it holds that $s_j^{(m)}\ne e_{G_m}$. 

Then we have the following convergence in $\mathcal{G}(2k)$:
\begin{eqnarray*}
& &(\mathrm{Sym}(G_m);\chi_{s_1^{(m)}}, \ldots ,\chi_{s_k^{(m)}},\theta_{s_1^{(m)}},\ldots ,\theta_{s_k^{(m)}}) \\
\quad \stackrel{\mathrm{Cay}}{\longrightarrow}\quad & &(\mathrm{Sym}_{<\aleph_0}(G_\infty)\rtimes G_{\infty};\chi_{s_1^{(\infty)}}, \ldots ,\chi_{s_k^{(\infty)}},\theta_{s_1^{(\infty)}},\ldots ,\theta_{s_k^{(\infty)}}).
\end{eqnarray*}

Here, $G_{\infty}$ acts on $\mathrm{Sym}_{<\aleph_0}(G_\infty)$ as permutations induced by right multiplication; for a countable group $G$ and for $\gamma \in G\setminus\{e_G\}$, we define  elements $\chi_{\gamma}\in \mathrm{Sym}_{<\aleph_0}(G)$ and $\theta_{\gamma} \in \mathrm{Sym}(G)$ by
\begin{eqnarray*}
\chi_{\gamma}&=&(\textrm{the transposition on $\{e_{G},\gamma\}$}), \\
\theta_{\gamma}&=&(\textrm{the permutation on $G$ given by the right-multiplication of $\gamma$}).
\end{eqnarray*}
\end{lemma}

\begin{proof}
First it holds that for every $m\in \mathbb{N}$, $(\chi_{s_1^{(m)}}, \ldots ,\chi_{s_k^{(m)}},\theta_{s_1^{(m)}},\ldots ,\theta_{s_k^{(m)}})$ is a marking of $\mathrm{Sym}(G_m)$; this is because for every $\gamma\in G_m\setminus \{e_{G_m}\}$, we may write the transposition on $\{e_{G_m},\gamma\}$ as a certain product of these elements. Similarly, we can show that $(\chi_{s_1^{(\infty)}}, \ldots ,\chi_{s_k^{(\infty)}},\theta_{s_1^{(\infty)}},\ldots ,\theta_{s_k^{(\infty)}})$ is a marking of $\mathrm{Sym}_{<\aleph_0}(G_{\infty})\rtimes G_{\infty}$.

Finally, we prove the Cayley convergence in the statement of Lemma~\ref{lemma=SymmetricGroups}; as we announced below Remark~\ref{remark=Gruenberg}, we only give some intuitive description. Let for $m\in \mathbb{N}$, $\mathbf{H}_m=(\mathrm{Sym}(G_m);\chi_{s_1^{(m)}}, \ldots ,\chi_{s_k^{(m)}},\theta_{s_1^{(m)}},\ldots ,\theta_{s_k^{(m)}})$ and let $\mathbf{H}_{\infty}=(\mathrm{Sym}_{<\aleph_0}(G_\infty)\rtimes G_{\infty};\chi_{s_1^{(\infty)}}, \ldots ,\chi_{s_k^{(\infty)}},\theta_{s_1^{(\infty)}},\ldots ,\theta_{s_k^{(\infty)}})$ The key here is the following: for every $m\in \mathbb{N}$ and for every $R\in \mathbb{R}_{\geq 0}$, there exists $r=r(m,R)\leq R$ with $r(m,R)\to \infty$ as $\min\{m,R\}\to \infty$ such that elements of $B_{\mathbf{H}_m}(e_{\mathrm{Sym}(G_m)},r)$ is completely determined by its image of $B_{\mathbf{G}_m}(e_{G_m},R)$; more precisely, the map
\[
 \mathrm{Sym}(G_m)\to \mathrm{Map}(B_{\mathbf{G}_m}(e_{G_m},R),G_m);\quad \xi \mapsto \xi\mid_{B_{\mathbf{G}_m}(e_{G_m},R)}
\]
is \textit{injective on $B_{\mathbf{H}_m}(e_{\mathrm{Sym}(G_m)},r)$}. A similar statement holds for $m=\infty$ as well. For such $r=r(m,R)$, it holds that for by every element of $B_{\mathbf{H}_m}(e_{\mathrm{Sym}(G_m)},r)$, the image of $B_{\mathbf{G}_m}(e_{G_m},R)$ by it is included in $B_{\mathbf{G}_m}(e_{G_m},R+r)$. Therefore, we conclude that the \textit{local} picture of $\mathrm{CayD}(\mathbf{G}_m)$ \textit{completely determines} that of $\mathrm{CayD}(\mathbf{H}_m)$. Hence, the Cayley convergence $\mathbf{G}_{m}\stackrel{\mathrm{Cay}}{\to}\mathbf{G}_{\infty}$ may be encoded into symmetric groups as $\mathbf{H}_{m}\stackrel{\mathrm{Cay}}{\to}\mathbf{H}_{\infty}$.
\end{proof}

We proceed to encoding into alternating groups. For each $m\in \mathbb{N}$ and each $j\in [k]$, $\chi_{s_j^{(m)}}$ above has a negative sign; the sign of $\theta_{s_j^{(m)}}$ equals the parity of the product of $((\textrm{the order of $s_j^{(m)}$})-1)$ and $\#\left(G_m/\langle s_j^{(m)}\rangle\right)$. To obtain our encoding result into alternating groups, we need to impose some condition on signs of $\theta_{s_j^{(m)}}$. Here is the result which we will use.

\begin{lemma}[Encoding into alternating groups]\label{lemma=AlternatingGroups}We stick to the setting of Lemma~$\ref{lemma=SymmetricGroups}$.  We furthermore assume that for each $m\in \mathbb{N}$ and each $j\in [k]$, $\theta_{s_j^{(m)}}$ has a positive sign. Then, 
\[
(\chi_{s_1^{(m)}}\chi_{s_2^{(m)}},\ldots ,\chi_{s_1^{(m)}}\chi_{s_k^{(m)}}, \theta_{s_1^{(m)}},\ldots , \theta_{s_k^{(m)}}, \chi_{s_1^{(m)}}\theta_{s_1^{(m)}}\chi_{s_1^{(m)}}, \ldots ,\chi_{s_1^{(m)}}\theta_{s_k^{(m)}}\chi_{s_1^{(m)}})
\]
is a $(3k-1)$-marking of $\mathrm{Alt}(G_m)$ for every $m\in \mathbb{N}$, and it is a $(3k-1)$-marking of $\mathrm{Alt}(G_{\infty})\rtimes G_{\infty}$ for $m=\infty$. Moreover, we have the Cayley convergence
\[
\mathrm{Alt}(G_m)\  \stackrel{\mathrm{Cay}}{\longrightarrow}\  \mathrm{Alt}(G_\infty)\rtimes G_{\infty}
\]
with respect to the $(3k-1)$-markings above for $m\in \mathbb{N}\cup \{\infty\}$.
\end{lemma}

\begin{proof}
It follows from Lemma~\ref{lemma=SymmetricGroups}, the Nielsen--Schreier algorithm on generators of finite index subgroups, and Lemma~\ref{lemma=LEFsubgroups}.
\end{proof}

For instance, the assumption in Lemma~\ref{lemma=AlternatingGroups} on signs is satisfied if $\#\left(G_m/\langle s_j^{(m)}\rangle \right)$ is an even number for all $m\in \mathbb{N}$ and for all $j\in [k]$.

We recall the following result of Ore \cite{Ore} on commutators of alternating groups, which will be used in the proof of Theorem~\ref{mtheorem=MainTheorem}; as we mentioned in Subsection~\ref{subsection=Hall}, it is of importance to write a certain element \textit{as a single commutator} to apply a Hall-type argument.
\begin{lemma}[\cite{Ore}]\label{lemma=Ore}
For $n\in \mathbb{N}_{\geq 5}$, every element in $\mathrm{Alt}([n])$ may be written as a single commutator.
\end{lemma}

\section{Proof of Theorem~\ref{mtheorem=MainTheorem}}\label{section=ProofOfMainTheorem}
Recall our outlined proof of Theorem~\ref{mtheorem=MainTheorem} from Section~\ref{section=Organization}. The second and fourth steps are already described, respectively, in Subsections~\ref{subsection=Encoding} (Lemma~\ref{lemma=AlternatingGroups}) and \ref{subsection=DiagonalProducts} (Lemma~\ref{lemma=DiagonalProducts}). The first step  will be explained in Subsection~\ref{subsection=Auxiliary} (Lemma~\ref{lemma=Auxiliary}). In the third step we employ Proposition~\ref{proposition=Absorption}; it is a combination of a Hall-type embedding argument (Lemma~\ref{lemma=Hall}) and the absorption trick (see Lemma~\ref{lemma=Absorption} for a prototype). We will explain Proposition~\ref{proposition=Absorption}, which  is the key proposition to the proof,  in Subsection~\ref{subsection=Absorption}. In Subsection~\ref{subsection=ProofOfMainTheorem}, we demonstrate the complete proof of Theorem~\ref{mtheorem=MainTheorem}.

\subsection{An auxiliary lemma}\label{subsection=Auxiliary}
We prove the following auxiliary lemma, which enables us to embed a finitely generated LEF group into a \textit{LEF} group generated by finitely many \textit{torsions}.

\begin{lemma}[Embedding into a LEF group generated by involutions]\label{lemma=Auxiliary}
Let $G$ be a finitely generated LEF group. Then there exists a finitely generated group $G^{\#}$ that satisfies the following three conditions.
\begin{itemize}
  \item $G^{\#}$ admits a set of generators all of whose elements are of order $2$.
  \item $G^{\#}$ contains an isomorphic copy of $G$. 
  \item $G^{\#}$ is LEF.
\end{itemize}
\end{lemma}

\begin{proof}
Let $G$ be a LEF group generated by $s_1^{(\infty)},\ldots ,s_k^{(\infty)}$. In what follows, we construct a group $G^{\#(1)}$, generated by $(k+1)$-elements $a,b,s_1^{\#(1)},\ldots ,s_{k-1}^{\#(1)}$, such that the following hold: 
\begin{itemize}
  \item $a$ and $b$ are torsions of order $2$. 
  \item For every $j\in [k-1]$, the order of $s_j^{\#(1)}$ coincides with that of $s_j^{(\infty)}$.  
  \item $G^{\#(1)}$ contains an isomorphic copy of $G$.
  \item $G^{\#(1)}$ is LEF.
\end{itemize}
Iteration of this procedure will yield a desired group $G^{\#}$.

Our construction goes as follows: For $n\in \mathbb{N}_{\geq 2}\cup\{\infty\}$, denote by $D_{n}=\langle c_n,d_n | c_n^2=d_n^2=(c_nd_n)^{n}=e_{D_{n}}\rangle$ the dihedral group of degree $n$ (if $n=\infty$, ignore the relation $(c_nd_n)^n=e_{D_{n}}$). Let $n_{\infty}\in \mathbb{N}_{\geq 2}\cup\{\infty\}$ be the order of $s_k$. Then, set
\[
G^{\#(1)}=D_{n_{\infty}}\ast_{\mathbb{Z}/n_{\infty}\mathbb{Z}}G,
\]
where we consider $\mathbb{Z}/\infty \mathbb{Z}$ to be $\mathbb{Z}$. 
Here the amalgam on the right hand side is taken with respect to homomorphisms $\mathbb{Z}/n_{\infty}\mathbb{Z}\hookrightarrow D_{n_{\infty}}$ and $\mathbb{Z}/n_{\infty}\mathbb{Z}\hookrightarrow G$ that sends $1\in \mathbb{Z}/n_{\infty}\mathbb{Z}$, respectively, to $c_{n_{\infty}}d_{n_{\infty}}$ and $s_k^{(\infty)}$. We set $a,b$ and $s_1^{\#(1)},\ldots ,s_{k-1}^{\#(1)}$ and  as the images of $c_{n_{\infty}},d_{n_{\infty}}$ and $s_1^{(\infty)},\ldots ,s_{k-1}^{(\infty)}$, respectively, by natural injections $D_{n_{\infty}}\hookrightarrow G^{\#(1)}$ and $G\hookrightarrow G^{\#(1)}$. 

We claim that this $G^{\#(1)}$ fulfills all of the four conditions above. Indeed, in what follows, we will show that $G^{\#(1)}$ is LEF; the other conditions are by construction. Take a LEF approximation $((G_m;s_1^{(m)},\ldots ,s_k^{(m)}))_{m\in \mathbb{N}}$ of $(G;s_1^{(\infty)},\ldots ,s_k^{(\infty)})$. For each $m\in \mathbb{N}$, let $l_m\in \mathbb{N}_{\geq 2}$ be the order of $s_k^{(m)}$. Consider the following sequence of $(k+1)$-marked groups,
\[
((D_{n_{m}}\ast_{\mathbb{Z}/n_{m}\mathbb{Z}}G_m; c_{l_m}, d_{l_m},s_1^{(m)},\ldots ,s_{k-1}^{(m)}))_{m\in \mathbb{N}}.
\]
Here the amalgam is taken with respect to homomorphisms $\mathbb{Z}/n_{m}\mathbb{Z}\hookrightarrow D_{n_{m}}$ and $\mathbb{Z}/n_{m}\mathbb{Z}\hookrightarrow G_m$ that sends $1\in \mathbb{Z}/n_{m}\mathbb{Z}$, respectively, to $c_{l_m}d_{l_m}$ and $s_k^{(m)}$. We indentify $D_{n_{m}}$ and $G_m$, respectively,  with the natural copies of them inside $D_{n_{m}}\ast_{\mathbb{Z}/n_{m}\mathbb{Z}}G_m$. Then, it can be seen that this sequence converges to $(G^{\#(1)}; a,b, s_1^{\#(1)},\ldots ,s_{k-1}^{\#(1)})$ (recall Lemma~\ref{lemma=LEFsubgroups}). Since for each $m\in \mathbb{N}$, $D_{n_{m}}$ and $G_m$ are finite, the group $D_{n_{m}}\ast_{\mathbb{Z}/n_{m}\mathbb{Z}}G_m$ is virtually free, that means, it contains a (possibly rank $1$) free group of finite index. Hence it is RF (and in particular, it is LEF). See also \cite[Theorem~3]{Baumslag}. Therefore, $G^{\#(1)}$ is LEF; recall our discussion in Section~\ref{section=Organization}. (Compare with arguments in \cite{Berlai}.)
\end{proof}
\begin{remark}\label{remark=RFtorsion}
The proof above of Lemma~\ref{lemma=Auxiliary} shows that if $G$ above is RF, then we can take $G^{\#}$ such that it is moreover RF.
\end{remark}

Burger and Mozes \cite{BurgerMozes} constructed an example of a group of the form $F\ast_E F$, where $F$ and $E$ are finitely generated free groups and $E$ is of finite index in both sides of $F$, such that it is finitely presented and simple. In particular, it is not LEF; recall Lemma~\ref{lemma=FinitelyPresented}. On the other hand, our proof above works even when $n_{\infty}=\infty$ because in our case, regardless of the group $\langle s^{(m)}_k\rangle$, we can adjust $D_{\infty}$ to its finite quotient $D_{n_{m}}$ and  have compatibility in order to obtain an amalgamated free product $D_{n_{m}}\ast_{\mathbb{Z}/n_{m}\mathbb{Z}}G_m$ in every finitary stage.

\subsection{The absorption trick}\label{subsection=AbsorptionTrick}

In Subsection~\ref{subsection=Hall}, we saw one important point to employ standard (restricted) wreath products: It helps us to reduce the number of generators of LEF groups. In this subsection, we explain another point, which is the most significant one, to make use of them. The author calls it \textit{the absorption trick}. This trick was at least observed in a work of Bartholdi and Erschler. We exhibit a form of the absorption trick, which is easily deduced from \cite[Lemma~6.13]{BartholdiErschler}.

\begin{lemma}[Prototype of the absorption trick; \cite{BartholdiErschler}]\label{lemma=Absorption}
Let $k\in \mathbb{N}_{\geq 1}$. Let $\mathbf{G}=(G;S)=(G;s_1,\ldots ,s_k)$ be a $k$-marked group. Then, there exists a sequence $(S_m)_{m\in \mathbb{N}}$ of $(k+1)$-markings of $G$ such that we have a convergence
\[
(G\wr \mathbb{Z};S_m)\ \stackrel{\mathrm{Cay}}{\to} \ (C_1\times C_2\times \cdots \times C_k)\wr \mathbb{Z}
\]
with respect to a certain marking of the Cayley limit group. Here for each $j\in [k]$, $C_j$ is the cyclic group with the same order as $s_j$.
\end{lemma}

We use the word ``\textit{absorption}'' because this lemma may be seen as ``\textit{absorbing} the original group $G$ into an abelian group by taking the standard (restricted) wreath product with $\mathbb{Z}$.''

\begin{proof}
Let $m\in \mathbb{N}$. Set $S_m=(s_1^{m)},s_2^{(m)},\ldots ,s_k^{(m)},t)$, where
\[
s_j^{(m)}=(z_j^{(m)},0)\quad \textrm{for $j\in [k]$}\quad \textrm{and}\quad t=(\mathbf{e},1).
\]
Here for each $j\in [k]$, define $z_j^{(m)}\in \bigoplus_{\mathbb{Z}}G$ by for $l\in \mathbb{N}$,
\[
z_j^{(m)}(n)=\left\{\begin{array}{cl}
s_j, & n=2^m (j-1), \\
e_G, & \textrm{otherwise}.
\end{array}\right.
\]
It is easy to see that $S_m$ is a marking of $G\wr \mathbb{Z}$. 

Finally, we describe an intuitive proof of the Cayley convergence of our concern. The key observation here is the following: If $m\in \mathbb{N}$ is sufficiently large compared with $R\in \mathbb{R}_{\geq 0}$, then \textit{inside $B_{(G\wr \mathbb{Z};S_m)}(e_{G\wr \mathbb{Z}},R)$, all these $k$ elements $s_1^{(m)},\ldots ,s_k^{(m)}$, and conjugations by powers of $t$, behave as if they were ``independent.''} We explain this assertion in more detail. For instance, we know that 
\[
\langle s_1^{(m)}, t^{2^m}s_2^{(m)}t^{-2^m}\rangle \simeq \langle s_1,s_2\rangle;
\]
this means that although $s_1^{(m)}$ and $s_2^{(m)}$ commute, after taking conjugations of certain powers of $t$, we eventually see that these two elements ``interact.'' In this way, in the \textit{global} picture, we observe that $s_1^{(m)},\ldots ,s_k^{(m)}$ interact to each other, via conjugations of powers of $t$ so that $s_1^{(m)},\ldots ,s_k^{(m)},t$ generate the whole $G\wr \mathbb{Z}$. On the other hand, in the \textit{local} picture, more precisely, if $R<2^{m}$, then inside the $R$-ball of $(G\wr\mathbb{Z};S_m)$ centered at $e_{G\wr \mathbb{Z}}$, we do \textit{not} see any interaction of $s_1^{(m)},\ldots ,s_k^{(m)}$ and their conjugations by powers of $t$; inside this $R$-ball, all of these elements \textit{commute}. Therefore, from the \textit{local} point of view, the $(k+1)$-marking $S_m=(s_1^{(m)},\ldots ,s_k^{(m)},t)$ becomes closer and closer to the marking 
\[
((c_1\delta_{0},0),(c_2\delta_{0},0),\ldots ,(c_k\delta_{0},0), (\mathbf{e},1)),
\]
where for each $j\in [k]$, $c_j$ is a  cyclic generators of the cyclic group $C_j$ of the same order as $s_j$; furthermore, these $c_1,\ldots,c_k$ are ``\textit{independent}'', which means that they all commute. Therefore, the Cayley limit  group is of the form $(C_1\times C_2\times \cdots \times C_k)\wr \mathbb{Z}$. It describes the Cayley convergence
\[
(G\wr \mathbb{Z};S_m)\  \stackrel{\mathrm{Cay}}{\longrightarrow}\  (C_1\times C_2\times \cdots \times C_k)\wr \mathbb{Z}
\]
with respect to the limit marking as above. 
\end{proof}
The proof above  explains utility of standard wreath products to \textit{create enough room for the absorption trick}.

\subsection{The key proposition: combination of the absorption trick and a Hall-type argument}\label{subsection=Absorption}

The third step in the outlined proof of Theorem~\ref{mtheorem=MainTheorem} described in Section~\ref{section=Organization} will be done by combination of the absorption trick as in Subsection~\ref{subsection=AbsorptionTrick} (Lemma~\ref{lemma=Absorption}) and a Hall-type argument as in Subsection~\ref{subsection=Hall} (Lemma~\ref{lemma=Hall}). The precise form of the result is Proposition~\ref{proposition=Absorption} below. It is worth mentioning again that formation of standard (unrestricted) wreath products enable us to perform both of the two arguments above \textit{at the same time}.

\begin{proposition}[Key proposition: combination of the absorption trick and a Hall-type argument]\label{proposition=Absorption}
Let $G_{\infty}$ be a finitely generated LEF group and $S_{\infty}=(s_1^{(\infty)},\ldots,s_k^{(\infty)})$ be a  marking of $G_{\infty}$. Then there exist
\begin{itemize}
  \item a sequence $(p_m)_{m\in \mathbb{N}}$ of strictly increasing prime numbers;
  \item a sequence of groups $(L_m)_{m\in \mathbb{N}}$, where for each $m\in \mathbb{N}$, $L_m$ is a subgroup of the standard wreath product $G_m\wr (\mathbb{Z}/p_m\mathbb{Z})$; and 
  \item three sequences of group elements $(w^{(m)})_{m\in \mathbb{N}}$, $(t^{(m)})_{m\in \mathbb{N}}$ and $(u^{(m)})_{m\in \mathbb{N}}$, where for each $m\in \mathbb{N}$, $w_m,t_m,u_m\in L_m$
\end{itemize}
such that all of the following conditions are satisfied.
\begin{enumerate}[$(1)$]
  \item For each $m\in \mathbb{N}$, $\langle w^{(m)},t^{(m)}\rangle=L_m$ and $\langle w^{(m)},u^{(m)}\rangle=L_m$.
  \item The sequence of $2$-marked groups $((L_m;w^{(m)},t^{(m)}))_{m\in \mathbb{N}}$ converges in the Cayley topology to a marked group with underlying group $\Gamma_1=C\wr\mathbb{Z}$, where $C$ is a cyclic group of the same order as that of $\bigoplus_{j\in [k]}s_j^{(\infty)} \delta_{j}(\in \bigoplus_{j\in [k]}G_{\infty})$.
  \item The sequence of $2$-marked groups $((L_m;w^{(m)},u^{(m)}))_{m\in \mathbb{N}}$ converges in the Cayley topology to a marked group with underlying group $\Gamma_2$. Here $\Gamma_2$ contains an isomorphic copy of the subgroup of $G_{\infty}$ generated by $\{[s_i^{(\infty)},s_j^{(\infty)}]:i,j \in [k]\}$.
\end{enumerate}
\end{proposition}

\begin{proof}
Take a LEF approximation $(\mathbf{G}_m)_{m\in \mathbb{N}}=((G_m;S_m))_{m}$  of $\mathbf{G}_{\infty}=(G_{\infty};S_{\infty})$. Set two sequences $(p_m)_{m\in \mathbb{N}}$ and $(p'_m)_{m\in \mathbb{N}}$ of natural numbers that satisfy
\begin{itemize}
  \item for all $m\in \mathbb{N}$, $2^k<p'_m <p_m$;
  \item $\lim_{m\to \infty}p'_m=\infty$ and $\lim_{m\to \infty}(p_m/p'_m)=\infty$; and 
  \item for each $m\in \mathbb{N}$, $p_m$ and $p'_m$ are coprime.
\end{itemize}
Note that only these three conditions are needed to prove the proposition;  for instance, a concrete example $(p_m,p'_m)=((3m+2)^24^k,(2m+1)3^k)$ works. We may also take $(p_m)_{m\in \mathbb{N}}$ as a strictly increasing sequence of primes. The latter example will be used in the proof of Theorem~\ref{mtheorem=MainTheorem}.

Given $(\mathbf{G}_m)_{m\in \mathbb{N}}$, $(p_m)_{m\in \mathbb{N}}$ and $(p'_m)_{m\in \mathbb{N}}$, we set for each $m\in \mathbb{N}$
\[
w^{(m)}=(f^{(m)},0),\ t^{(m)}=(\mathbf{e},p'_m) \textrm{ and } u^{(m)}=(\mathbf{e},1) \quad  (\textrm{in }G_m\wr (\mathbb{Z}/p_m\mathbb{Z})).
\]
Here  $f^{(m)}$ is defined by for $n\in \mathbb{Z}/p_m\mathbb{Z}$,
\[
f^{(m)}(n)=\left\{
\begin{array}{cl}
s_j^{(m)}, & \textrm{if $n=2^j$ for $j\in [k]$},\\
e_{G_m}, & \textrm{otherwise.}
\end{array}
\right.
\]
Finally, set $L_m$ as the subgroup of $G_m\wr (\mathbb{Z}/p_m\mathbb{Z})$ generated by $w^{(m)}$ and $u^{(m)}$. 

We claim that this construction meets all of the three assertions $(1)$--$(3)$ in the proposition. Since $p'_m$ and $p_m$ are coprime, $(1)$ holds true. In what follows, we present intuitive descriptions of the proofs of $(2)$ and $(3)$, which go along a similar line to one in the proof of Lemma~\ref{lemma=Absorption}; these arguments can be formalized to a rigorous proof. 

The key here is the following \textit{extreme difference} between two $2$-markings $(w^{(m)},t^{(m)})$ and $(w^{(m)},u^{(m)})$ \textit{in the local picture}: For the $2$-marked group $(G_m\wr (\mathbb{Z}/p_m\mathbb{Z});w^{(m)},u^{(m)})$, we may \textit{see many interactions} of $w^{(m)}$ and its conjugates by powers of $u^{(m)}$ even in the local picture. This is because if our $R$ is more than $2^{k+1}+1$, then for all $j\in [k]$, the element $(u^{(m)})^{-2^j}w^{(m)}(u^{(m)})^{2^j}$ is in the $R$-ball centered at $e_{G_m\wr (\mathbb{Z}/p_m\mathbb{Z})}$; these elements together with $w^{(m)}$ possess much information of the group $G_m$ \textit{even in the local picture}; we come back to this point later. In contrast, for the $2$-marked group $(G_m\wr (\mathbb{Z}/p_m\mathbb{Z});w^{(m)},t^{(m)})$, from the \textit{local} point of view, $w^{(m)}$ and its conjugates by powers of $t^{(m)}$ behave as if there were ``\textit{independent}.'' This follows from our way of defining $(p_m)_m$ and $(p'_m)_m$; in other words, in the local picture, supports of conjugates of $w^{(m)}$ by powers of $t^{(m)}$ are all disjoint. Compare with the proof of Lemma~\ref{lemma=Absorption}. 

From these key observations, we may determine the Cayley limit groups of $((G_m\wr (\mathbb{Z}/p_m\mathbb{Z});w^{(m)},t^{(m)}))_m$ and $((G_m\wr (\mathbb{Z}/p_m\mathbb{Z});w^{(m)},u^{(m)}))_m$, respectively, as follows.
\begin{itemize}
 \item \underline{(Absorption trick part):}\ For the former sequence $((G_m\wr (\mathbb{Z}/p_m\mathbb{Z});w^{(m)},t^{(m)}))_m$, \textit{the absorption trick applies.} What we obtain is the following Cayley convergence:
\[
(L_m;w^{(m)},t^{(m)})\  \stackrel{\mathrm{Cay}}{\longrightarrow}\  (\Gamma_1;w^{(\infty)},t^{(\infty)}).
\]
where $w^{(\infty)}$ and $t^{(\infty)}$ are elements in $G_{\infty}\wr \mathbb{Z}^2$ defined as follows: The element $w^{(\infty)}$ is $(f^{(\infty)},0)$, where $f^{(\infty)}$ is the element on $\bigoplus_{\mathbb{Z}^2}G_{\infty}$ defined by for $(n_1,n_2)\in \mathbb{Z}^2$,
\[
f^{(\infty)}((n_1,n_2))=\left\{
\begin{array}{cl}
s_j^{(\infty)}, & \textrm{if $n_1=2^j$ for $j\in [k]$ and if $n_2=0$},\\
e_{G_{\infty}}, & \textrm{otherwise.}
\end{array}
\right.
\]
The element $t^{(\infty)}$ is $(\mathbf{e},(0,1))$. The group $\Gamma_1$ is the subgroup of $G_{\infty}\wr \mathbb{Z}^2$ generated by $w^{(\infty)}$ and $t^{(\infty)}$. To sum up, the absorption trick enables us to obtain $t^{(\infty)}$ to be a shift on a direction that is \textit{independent from the direction in which non-neutral elements of $f^{(\infty)}$ lie}. Therefore, \textit{all} conjugates of $w^{(\infty)}$ by powers of $t^{(\infty)}$ \textit{commute}. By letting $C$ be a cyclic group of the same order as $w^{(\infty)}$, we hence have the marked group isomorphism
\[
(\Gamma_1;w^{(\infty)},t^{(\infty)})\  \cong \ (C\wr \mathbb{Z}; c,s),
\]
where $c$ is a generator of $C$ and $s$ is the  shift by $+1$ on $\mathbb{Z}$. This proves $(2)$.
\item \underline{(Hall-type argument part):}\ For the latter sequence $((G_m\wr (\mathbb{Z}/p_m\mathbb{Z});w^{(m)},u^{(m)}))_m$, in contrast to the former case, we see many interactions of $w^{(m)}$ and its conjugates by powers of $u^{(m)}$ even in the local picture. From this, we may conclude that
\[
(L_m;w^{(m)},u^{(m)})\  \stackrel{\mathrm{Cay}}{\longrightarrow}\  (\Gamma_2;w^{(\infty)},u^{(\infty)}).
\]
Here $w^{(\infty)}\in G_{\infty}\wr \mathbb{Z}^2$ is the same elements as above and $u^{(\infty)}=(\mathbf{e},(1,0))$; $\Gamma_2$ is the subgroup of $G_{\infty}\wr \mathbb{Z}^2$ generated by $w^{(\infty)}$ and $u^{(\infty)}$. The main difference from the former case is that the direction of the shift by $u^{(\infty)}$ is the \textit{same} as one in which non-neutral elements of $f^{(\infty)}$ lie. Hence, we can see many interactions of conjugates of $w^{(\infty)}$ by powers of $u^{(\infty)}$. In this situation, a similar argument to the Hall-type argument as in the proof of Lemma~\ref{lemma=Hall} applies, which makes use of commutators of these elements above. It is now straightforward to confirm $(3)$. 
\end{itemize}
\end{proof}

\subsection{Proof of Theorem~\ref{mtheorem=MainTheorem}}\label{subsection=ProofOfMainTheorem}
We split the proof into three parts: The proof of the main assertions ($(1)$--$(3)$), that of the assertion to take $t=u^3$,  and that of the last assertion on $H$. Recall our definition of $\theta_{\gamma}$ from Lemma~\ref{lemma=SymmetricGroups}. The proof makes full use of Cayley convergences in the space of marked groups. To comprehend the proof, the reader is supposed to be ready to obtain a new Cayley convergent sequence out of one by applying  Lemma~\ref{lemma=AlternatingGroups} and Proposition~\ref{proposition=Absorption}; he/she may also need to be accustomed to showing generation of a subgroup by certain elements in a standard wreath product by commutator calculus in a Hall-type argument,  similar to one in Lemma~\ref{lemma=Hall}. Recall arguments in  Lemma~\ref{lemma=WreathProducts}, Remark~\ref{remark=Gruenberg} and Lemma~\ref{lemma=Absorption} to grasp the \textit{local} point of view  for Cayley convergences.

\begin{proof}[Proof of the main assertions of Theorem~$\ref{mtheorem=MainTheorem}$]

We follow the outlined proof as in Section~\ref{section=Organization} by steps. Let $G$ be a finitely generated LEF group.

\begin{enumerate}
 \item[\textit{Step~$1$}.] By  Lemma~\ref{lemma=Auxiliary}, our $G$ embeds $G^{\#}$ that admits a marking $(s_1^{(\infty)},\ldots ,s_{2k}^{(\infty)})$ such that for every $j\in [2k]$, $s_j^{(\infty)}$ is of order $2$. Set $\mathbf{G}_{\infty}=(G^{\#};s_1^{(\infty)},\ldots ,s_{2k}^{(\infty)})$. 

Take a LEF approximation $(\mathbf{G}_m)_{m\in \mathbb{N}}$$=((G_m;s_1^{(m)},\ldots ,s_{2k}^{(m)}))_{m}$ of $\mathbf{G}_{\infty}$. Here we may assume that $\#(G_m)$ is at least $5$, strictly increasing with $m$, and divisible by $4$. (The last assertion follows from orders of generators for sufficiently large $m$; even if it is not the case, replace $G^{\#}$ with $G^{\#}\times  D_4$ and $G_m$ with $G_m\times D_4$ and add two more generators from $D_4$ which are involutions.)
 \item[\textit{Step~$2$}.] Apply Lemma~\ref{lemma=AlternatingGroups} to encode information of this convergence $\mathbf{G}_m\stackrel{\mathrm{Cay}}{\to} \mathbf{G}_{\infty}$ into alternating groups. Here note that since $\#(G_m)$, $m\in \mathbb{N}$, is divisible by $4$, $\theta_{\gamma}$ for $\gamma$ in the marking of $G_m$ all have positive signs. Thus, we obtain a $(6k-1)$-marking $(\xi_1^{(m)},\ldots,\xi_{6k-1}^{(m)})$ of $\mathrm{Alt}(G_m)$ for every $m\in \mathbb{N}$ and that $(\xi_1^{(\infty)},\ldots,\xi_{6k-1}^{(\infty)})$ of $\mathrm{Alt}(G^{\#})\rtimes G^{\#}$ such that
\[
(\mathrm{Alt}(G_m);\xi_1^{(m)},\ldots,\xi_{6k-1}^{(m)})
\  \stackrel{\mathrm{Cay}}{\longrightarrow}\  (\mathrm{Alt}(G^{\#})\rtimes G^{\#};\xi_1^{(\infty)},\ldots,\xi_{6k-1}^{(\infty)})
\]
Note that for every $m\in \mathbb{N}$, each $\xi_j^{(m)}$, $j\in [6k-1]$, is either of order $2$ or of order $3$ by construction.
 \item[\textit{Step~$3$}.] Apply Lemma~\ref{lemma=Ore} to $\mathrm{Alt}(G_m)$. Then  for $j\in [6k-1]$, we obtain $\eta_j^{(m)}$ and $\zeta_j^{(m)}$ in $\mathrm{Alt}(G_m)$ such that $\xi_j^{(m)}=[\eta_j^{(m)},\zeta_j^{(m)}]$. In fact, it follows from the proof of Lemma~\ref{lemma=Ore} in \cite{Ore} that we may take $\eta_j^{(m)}$ and $\zeta_j^{(m)}$ such that they are, respectively, either of order $2$ or order $3$. Now, consider a sequence of $(18k-3)$-marked groups 
\[
((\mathrm{Alt}(G_m);\xi_1^{(m)},\ldots,\xi_{6k-1}^{(m)},\eta_1^{(m)},\ldots,\eta_{6k-1}^{(m)},\zeta_1^{(m)},\ldots,\zeta_{6k-1}^{(m)}))_{m\in \mathbb{N}}.
\]
This sequence itself is not necessarily a converge sequence, but by compactness of $\mathcal{G}(18k-3)$, it admits a convergent subsequence in the Cayley topology. Hence, after passing to a subsequence, we may assume that the sequence above is a convergent sequence in the Cayley topology. 

Now apply the key proposition, Proposition~\ref{proposition=Absorption}, to this LEF approximation. We obtain $(p_m)_{m\in \mathbb{N}}$, $(L_m)_{m\in \mathbb{N}}$, $(w^{(m)})_{m\in \mathbb{N}}$, $(t^{(m)})_{m\in \mathbb{N}}$ and $(u^{(m)})_{m\in \mathbb{N}}$ as there. One of the keys to our proof of Theorem~\ref{mtheorem=MainTheorem} is that, in fact, for every $m\in \mathbb{N}$
\[
L_m=\mathrm{Alt}(G_m)\wr (\mathbb{Z}/p_m\mathbb{Z})
\]
holds true. Importance of it is in relation to density of the resulting $\Lambda_1$ and $\Lambda_2$; see Step~$4$ below.
To see the equality above, first apply a Hall-type argument (as in the proof of Lemma~\ref{lemma=Hall}) to see that for every $j\in [6k-1]$, the element $\xi_j^{(m)}(=[\eta_j^{(m)},\zeta_j^{(m)}])$ in the $0$-th coordinate of $\bigoplus_{(\mathbb{Z}/p_m\mathbb{Z})}\mathrm{Alt}(G_m)$ belongs to $L_m$. Since $(\xi_1^{(m)},\ldots,\xi_{6k-1}^{(m)})$ is a marking of $\mathrm{Alt}(G_m)$, we verify the equality above.

In summary, we have obtained two systems of $2$-markings $((w^{(m)},t^{(m)}))_{m\in \mathbb{N}}$ and $((w^{(m)},u^{(m)}))_{m\in \mathbb{N}}$ of $(L_m)_{m\in \mathbb{N}}=(\mathrm{Alt}(G_m)\wr (\mathbb{Z}/p_m{\mathbb{Z}}))_{m\in \mathbb{N}}$ such that the following hold true:
\begin{itemize} 
  \item $((L_m;w^{(m)},t^{(m)}))_{m\in \mathbb{N}}$ converges in the Cayley topology to a marked group with underlying group $\Gamma_1$. Here $\Gamma_1$ is of the form $C\wr \mathbb{Z}$ and $C$ is \textit{finite} cyclic  by Step~2. (More precisely, $C=\mathbb{Z}/6\mathbb{Z}$ in our construction; some modification enables us to take $C$ to be $\mathbb{Z}/2\mathbb{Z}$.)
  \item $((L_m;w^{(m)},u^{(m)}))_{m\in \mathbb{N}}$ converges in the Cayley topology to a marked group with underlying group $\Gamma_2$. Here $\Gamma_2$ contains an isomorphic copy of 
\[
\langle \xi_1^{(\infty)},\ldots,\xi_{6k-1}^{(\infty)}\rangle \quad (=\mathrm{Alt}(G^{\#})\rtimes G^{\#});
\]
in particular, $\Gamma_2$ contains an isomorphic copy of the original $G$.
\end{itemize}
 \item[\textit{Step~$4$}.] Take the diagonal products, respectively, of the two sequences of marked groups $((L_m;w^{(m)},t^{(m)}))_{m\in \mathbb{N}}$ and $((L_m;w^{(m)},u^{(m)}))_{m\in \mathbb{N}}$. Let
\begin{eqnarray*}
 (\Lambda_1;w,t)&=&\Delta_{m\in \mathbb{N}}((L_m;w^{(m)},t^{(m)})) \quad  \textrm{and}\\
  (\Lambda_2;w,u)&=&\Delta_{m\in \mathbb{N}}((L_m;w^{(m)},u^{(m)})).
\end{eqnarray*}
and 
\[
K=\prod_{m\in \mathbb{N}}L_m\quad \left(\simeq \prod_{m\in \mathbb{N}}(\mathrm{Alt}([l_m])\wr (\mathbb{Z}/p_m\mathbb{Z}))\right),
\]
where we set $l_m=\# (G_m)$ for every $m\in \mathbb{N}$. Apply Lemma~\ref{lemma=DiagonalProducts}; we then have assertions $(1)$ and $(2)$ from Step~3. Indeed, for $(1)$, note that $\Lambda_1$ is an LFNF-lift of $C\wr \mathbb{Z}$; in particular, it is a locally-finite-lift of $\mathbb{Z}$ because $C$ is finite. 

Finally, we prove $(3)$ (density). Recall that $l_m\geq 5$ and $p_m$ is a prime for every $m\in \mathbb{N}$. Since $L_m\simeq (\mathrm{Alt}([l_m])\wr (\mathbb{Z}/p_m\mathbb{Z})$, the set of 
composition factors of $L_m$ is $\{\mathrm{Alt}([l_m]),\mathbb{Z}/p_m\mathbb{Z}\}$. Because $l_m$ and $p_m$ are respectively strictly increasing with $m\in \mathbb{N}$, each finite simple group appears as a composition factor of $L_m$ \textit{for at most one $m\in \mathbb{N}$}. Therefore, Lemma~\ref{lemma=Goursat} applies, and we conclude that $\Lambda_1$ and $\Lambda_2$ are both dense in $K$. This completes our proof.
\end{enumerate}
\end{proof}

\begin{proof}[Proof of the assertion of Theorem~$\ref{mtheorem=MainTheorem}$ to take $t=u^3$]

We modify the construction above in order to take $t=u^3$. To do this, we first employ dihedral groups $D_{p_m}$, instead of $\mathbb{Z}/p_m\mathbb{Z}$. Given $G$, take Steps~1 and 2 above. Proceed to the former half of Step~3 above to obtain $(\mathrm{Alt}(G_m);\xi_1^{(m)},\ldots ,\xi_{6k-1}^{(m)},\eta_1^{(m)},\ldots ,\eta_{6k-1}^{(m)},\zeta_1^{(m)},\ldots ,\zeta_{6k-1}^{(m)})$. Rename $(\xi_1^{(m)},\ldots , \zeta_{6k-1}^{(m)})$ as $(\xi_1^{(m)},\ldots ,\xi_{18k-3}^{(m)})$. Fix a sequence of strictly increasing primes $(p_m)_{m\in \mathbb{N}}$ such that $p_0> 2^{30k}$. For every $m\in \mathbb{N}$, let 
\[
J_m=\mathrm{Alt}(G_m) \wr D_{p_m}
\]
and $y^{(m)}=(g^{(m)},e_{D_{p_m}})$, $a^{(m)}=(\mathbf{e},c_{p_m})$ and $b^{(m)}=(\mathbf{e},d_{p_m})$. Here $c_{p_m}$ and $d_{p_m}$ are standard two generators of $D_{p_m}$ of order $2$ (see the proof of Lemma~\ref{lemma=Auxiliary}) and 
\[
g^{(m)}(\gamma)=\left\{
\begin{array}{cl}
\xi_j^{(m)}, & \textrm{if $\gamma=((c_{p_m}d_{p_m})^2)^{2^j}$ for $j\in [18k-3]$},\\
e_{\mathrm{Alt}(G_m)}, & \textrm{otherwise.}
\end{array}
\right.
\]
Then, similar to the proof of Proposition~\ref{proposition=Absorption}, we have that $(y^{(m)},a^{(m)},b^{(m)})$ is a marking of $J_m$. Here, the reason for switching from $\mathbb{Z}/p_m\mathbb{Z}$ to $D_{p_m}$ is to obtain markings of \textit{fixed finite order}. (In this case, $a^{(m)}$ and $b^{(m)}$ are of order $2$. As in a remark in the proof of Theorem~\ref{mtheorem=MainTheorem}, we may arrange our construction in such a way that $y^{(m)}$ is of order $2$ as well.)

Now, we apply Lemma~\ref{lemma=AlternatingGroups} once more, this time to the marked groups $((J_m;y^{(m)},a^{(m)},b^{(m)}))_m$ to obtain a system of $8$-marking $(\mu_1^{(m)},\ldots,\mu_8^{(m)})$ of $(\mathrm{Alt}(J_m))_m$. Then by Lemma~\ref{lemma=Ore}), we may go the same line as the first half of Step~3 as follows; after passing to a subsequence if necessary, we obtain a convergent sequence of $24$-marked groups
\[
((\mathrm{Alt}(J_m);\mu_1^{(m)},\ldots,\mu_8^{(m)},\nu_1^{(m)},\ldots ,\nu_8^{(m)},\upsilon_{1}^{(m)},\ldots ,\upsilon_{8}^{(m)}))_{m\in \mathbb{N}},
\]
where for each $j\in [8]$, $[\nu_j^{(m)},\upsilon_j^{(m)}]=\mu_j^{(m)}$ holds.

Rename $(\mu_1^{(m)},\ldots ,\upsilon_{8}^{(m)})$ as $(\mu_1^{(m)},\ldots ,\mu_{24}^{(m)})$. By changing orders inside markings, we can have that $24$-marking such that
\[
\mu_{22}^{(m)}=\theta_{y^{(m)}},\ \mu_{23}^{(m)}=\theta_{a^{(m)}}\quad \mathrm{and}\quad \mu_{24}^{(m)}=\theta_{b^{(m)}}.
\]

For every $m\in \mathbb{N}$, take two elements $w^{(m)}$ and $u^{(m)}$ in $L_m=\mathrm{Alt}(J_m)\wr (\mathbb{Z}/p_m\mathbb{Z})$ as $w^{(m)}=(f^{(m)},0)$ and $u^{(m)}=(\mathbf{e},2^{22})$, where 
\[
f^{(m)}(n)=\left\{
\begin{array}{cl}
\mu_j^{(m)}, & \textrm{if $n=2^j$ for $j\in [23]$},\\
\mu_{24}^{(m)}, & \textrm{if $n=3\cdot 2^{23}$},\\
e_{\mathrm{Alt}(J_m)}, & \textrm{otherwise.}
\end{array}
\right.
\]
Here we use the same sequence $(p_m)_m$ of primes as one in the argument above in the proof of Theorem~\ref{mtheorem=MainTheorem}. 
In what follows, we investigate, respectively, these two systems of $2$-marked groups $((L_m;w^{(m)},(u^{(m)})^3))_m$ and $((L_m;w^{(m)},u^{(m)}))_m$.
\begin{itemize}
\item \underline{(Absorption trick part):}\ Here we discuss the marking $((L_m;w^{(m)},(u^{(m)})^3))_m$. Observe that $(u^{(m)})^3=(\mathbf{e},3\cdot 2^{22})$. From this, it is easy to see by construction that in the local picture, all conjugates of $w^{(m)}$ by powers of $(u^{(m)})^3$ \textit{does not interact}. By the absorption trick, we conclude that, the underlying group $\Gamma_1$ of the Cayley limit of $((L_m;w^{(m)},t^{(m)}))_m$ is of the form $C\wr \mathbb{Z}$, where $C$ is finite cyclic. This $C$ may be taken as $\mathbb{Z}/2\mathbb{Z}$.
\item \underline{(Hall-type argument part):}\ Unlike the former case, by our construction, $w^{(m)}$ and $(u^{(m)})^4 w^{(m)}(u^{(m)})^{-4}$ \textit{does interact}; \textit{this interaction survives even in the local picture}. More precisely, we have  that for each $m$,
\[
u^{(m)}[w^{(m)},(u^{(m)})^4 w^{(m)}(u^{(m)})^{-4} ](u^{(m)})^{-1}=(\theta_{(a^{(m)}b^{(m)})^2}\delta_{2^{22}},0).
\]
Hence the element in the left hand side of the equality above does interact with $w^{(m)}$ as well. 
Write the element in the equality above as $q^{(m)}$. Now we apply the following Hall-type argument: For every $i,j\in [18k-3]$, we have that
\begin{eqnarray*}
& &[(q^{(m)})^{2^i}w^{(m)}(q^{(m)})^{-2^i},(q^{(m)})^{2^j}w^{(m)}(q^{(m)})^{-2^j}] \\
&=& ([\theta_{((a^{(m)}b^{(m)})^2)^{2^i}y^{(m)}((a^{(m)}b^{(m)})^2)^{-2^i}},\theta_{((a^{(m)}b^{(m)})^2)^{2^j}y^{(m)}((a^{(m)}b^{(m)})^2)^{-2^j}}] \delta_{2^{22}},0)\\
&=& (\theta_{([\xi_i^{(m)},\xi_j^{(m)}]\delta_{e_{D_{p_m}}},e_{D_{p_m}})}\delta_{2^{22}},0).
\end{eqnarray*}
By our construction of $\xi_i^{(m)}$, $i\in [18k-3]$, it implies that $\Gamma_2$ contains an isomorphic copy of $\mathrm{Alt}(G^{\#})\rtimes G^{\#}$. 
\end{itemize}
Set $l_m=\#(J_m)$ and form the  diagonal products associated with, respectively, the two systems of $2$-marked groups above; Lemma~\ref{lemma=DiagonalProducts} ends our proof.
\end{proof}

\begin{proof}[Proof of the last assertion of Theorem~$\ref{mtheorem=MainTheorem}$ on $H$]
Apply Lemma~\ref{lemma=Auxiliary}, and then we obtain a finitely generated \textit{RF} group $H^{\#}$ generated by elements of order $2$; see Remark~\ref{remark=RFtorsion}. Since $H^{\#}$ is RF, we can take a LEF approximation $(H_m)_{m\in \mathbb{N}}=(H^{\#}/N_m)_{m}$ of $H^{\#}$ (with respect to appropriate markings) coming from a chain $(N_m)_{m\in \mathbb{N}}$ of normal subgroups of $H^{\#}$; recall Subsection~\ref{subsection=Profinite}. In a similar way to one in the proof above, we may assume that each $\theta_{\gamma}$ for $\gamma$ in the marking of $H_m$ has a positive sign.

In our construction in the proof of the rest of Theorem~\ref{mtheorem=MainTheorem}, replace $G^{\#}$ with $G^{\#}\times H^{\#}$; replace a LEF approximation $(G_m)_{m\in \mathbb{N}}$ with $(G_m\times H_m)_{m\in \mathbb{N}}$ with standard markings associated to direct products of two marked groups. Obtain two sequences of $2$-marked groups $((L_m;w^{(m)},t^{(m)}))_{m\in \mathbb{N}}$ and $((L_m;w^{(m)},u^{(m)}))_{m\in \mathbb{N}}$, where $t^{(m)}=(u^{(m)})^3$, in this setting. 

We finally claim that the underlying group $\Lambda_2$ of $\Delta_{m\in \mathbb{N}}((L_m;w^{(m)},u^{(m)}))$ contains an isomorphic copy of $H^{\#}$. To show this, first recall from Subsection~\ref{subsection=Profinite} that 
\[
(\Lambda_2,w,u)(=\Delta_{m\in \mathbb{N}}((L_m;w^{(m)},u^{(m)})))\  \cong \  \Delta_{n\in \mathbb{N}}((\Lambda^{(\mathbb{M}_n)};w^{(\mathbb{M}_n)},u^{(\mathbb{M}_n)})).
\]
Write the marking $(h_1^{(\infty)},\ldots ,h_{\ell}^{(\infty)})$ of $H^{\#}$ which was used to construct $w$ and $u$ above. For every $m\in \mathbb{N}$, set $h_j^{(m)}=h_j^{(\infty)}\ \mathrm{mod}\ N_m$ in $H_m(=H^{\#}/N_m)$ for $j\in [\ell]$. Then, for every $m\in \mathbb{N}$, we have a group isomorphism
\[
\langle \theta_{\theta_{h_1^{(m)}}},\ldots ,\theta_{\theta_{h_{\ell}^{(m)}}}\rangle \simeq H_m,
\]
where for $j\in[\ell]$, $\theta_{\theta_{h_j^{(m)}}}$ is regarded as an element in $\mathrm{Alt}(\mathrm{Alt}(G_m\times H_m)\wr D_{p_m})$. Similarly, for $j\in[\ell]$ we consider $\theta_{\theta_{h_j^{(\infty)}}}$ as an element in $\mathrm{Sym}(\mathrm{Sym}(G^{\#}\times H^{\#})\wr D_{\infty})$. Then, by our construction of $(H_m)_{m\in \mathbb{N}}$, we have the following marked group isomorphism:
\[
\Delta_{n\in \mathbb{N}}(\Delta_{m\in \mathbb{M}_n}((H^{(m)};\theta_{\theta_{h_1^{(m)}}},\ldots ,\theta_{\theta_{h_{\ell}^{(m)}}})))\ \cong \ (H^{(\mathbb{N})};\theta_{\theta_{h_1^{(\infty)}}},\ldots ,\theta_{\theta_{h_{\ell}^{(\infty)}}});
\]
see Subsection~\ref{subsection=Profinite}. Here for every $m\in \mathbb{N}$, the group $H^{(m)}$ is defined as the subgroup of $\mathrm{Alt}(\mathrm{Alt}(G_m\times H_m)\wr D_{p_m})$ generated by $\theta_{\theta_{h_1^{(m)}}},\ldots ,\theta_{\theta_{h_{\ell}^{(m)}}}$; as an abstract group, it is isomorphic to $H_m$. The group $H^{(\mathbb{N})}$ is defined as the underlying group of the diagonal product $\Delta_{m\in \mathbb{N}}((H^{(m)};\theta_{\theta_{h_1^{(m)}}},\ldots ,\theta_{\theta_{h_{\ell}^{(m)}}}))$. The underlying group of the diagonal product in the left hand side above is a subgroup of  that of $\Delta_{n\in \mathbb{N}}((\Lambda^{(\mathbb{M}_n)};w^{(\mathbb{M}_n)},u^{(\mathbb{M}_n)}))$, which equals $\Lambda_2$. Since $H^{(m)}\simeq H_m$ for every $m\in \mathbb{N}$, it follows that $H^{(\mathbb{N})}\simeq H^{\#}$. Therefore, $\Lambda_2$ contains an isomorphic copy of $H^{\#}(\geqslant H)$, as desired. 
\end{proof}

\begin{remark}\label{remark=tandu}
The proof above (to take $t=u^3$) implies that for each $n\in \mathbb{N}_{\geq 3}$, we may arrange as $t=u^n$ in the statement of Theorem~\ref{mtheorem=MainTheorem}. To prove this for $n\in \{4,5\}$, we need to modify our construction of $w^{(m)}$ accordingly.
\end{remark}

\begin{remark}\label{remark=LFNF-lifts2}
In fact, the proof of Theorem~\ref{mtheorem=MainTheorem} implies that the group $\Lambda_1$ appearing in Theorem~\ref{mtheorem=MainTheorem} may be taken to be an LFNF-lift of $C\wr \mathbb{Z}$, where $C$ is finite cyclic (we may furthermore take $C=\mathbb{Z}/2\mathbb{Z}$). 

In the forthcoming work of R. Tanaka and the author, we study permanence properties for certain LFNF-lifts. As a byproduct of it, we show that the group $\Lambda_1$ appearing in Theorem~\ref{mtheorem=MainTheorem} may be arranged such that it has the \textit{Liouville property} (with respect to every symmetric finite generating set) and Shalom's \textit{property} $H_{\mathrm{FD}}$. We do not recall the definitions of these properties; see \cite{BrieusselZheng}, \cite{ShalomHFD} and \cite{BrieusselZhengHFD} for the definitions and more details. We make a remark that among amenable groups, groups with these properties are, respectively, considered as ``small'' groups in certain senses. For instance, it is known by \cite[Proposition 6.4]{KaimanovichVershik} that $(\mathbb{Z}/2\mathbb{Z})\wr \mathbb{Z}^3$ does not have the Liouville property; in \cite[Proposition~4.4]{BrieusselZhengHFD}, it is showed that $\mathrm{Sym}_{<\aleph_0}(\mathbb{Z})\rtimes \mathbb{Z}$ fails to have property $H_{\mathrm{FD}}$.
\end{remark}

\section{Proof of Theorem~\ref{mtheorem=SpectralGap}}\label{section=EmbeddingWithControl}

\subsection{Encoding into special linear group}\label{subsection=EncodingSpecialLinear}

Let $B$ be a non-empty at most countable set. Let $R$ be an associative ring with unit, possibly non-commutative. Consider the semigroup of all matrices $(a_{i,j})_{i,j\in B}$ over $R$ such that $a_{i,j}=0$ all but finitely many $j$ for every fixed $i$ and that $a_{i,j}=0$ all but finitely many $i$ for every fixed $j$. This is in fact a monoid with unit $I$ (the identity matrix). We define $\mathrm{GL}(B,R)$ as the group of all invertible elements of this monoid. If moreover $R$ is commutative, then by $\mathrm{SL}(B,R)$, we denote the union of $\mathrm{SL}(B_0,R)(=\{g\in \mathrm{Mat}_{B_0\times B_0}(R): \mathrm{det}(g)=1\})$ over all non-empty finite subsets $B_0\subseteq B$ via the natural inclusion $\mathrm{GL}(B_0,R) \hookrightarrow \mathrm{GL}(B,R)$. For $i,j \in B$ with $i\ne j$ and $r\in R$, we define an \textit{elementary matrix} $e_{i,j}^a$ by
\begin{eqnarray*}
(e_{i,j}^r)_{k,l}=
\left\{
\begin{array}{cll}
1, & \quad \textrm{for}\quad k=l,\\
r, & \quad \textrm{for}\quad (k,l)=(i,j),\\
0, & \quad \textrm{otherwise}.
\end{array}
\right.
\end{eqnarray*}
This is an element in $\mathrm{GL}(B,R)$ (it is in $\mathrm{SL}(B,R)$ if $R$ is commutative).

Let $(\mathbf{G}_{m})_{m\in \mathbb{N}}=((G_m;s_1^{(m)},\ldots ,s_k^{(m)}))_{m}$ be a LEF approximation of an infinite group $\mathbf{G}_{\infty}=(G_{\infty};s_1^{(\infty)},\ldots ,s_k^{(\infty)})$. Without loss of generality, we assume that for every $m\in \mathbb{N}\cup\{\infty\}$ and for every $j\in [k]$, $s_j^{(m)}\ne e_{G_m}$ holds. Similar to the case of alternating groups in Lemma~\ref{lemma=AlternatingGroups}, we now assume that for each $m\in \mathbb{N}$ and each $j\in [k]$, the element $\theta_{s_j^{(m)}}\in \mathrm{Sym}(G_m)$ has a positive sign. For a countable group $G$, for $\gamma\in G\setminus \{e_G\}$ and for a prime $p$, define elements $\sigma_{\gamma}=\sigma_{\gamma}(p)$ and $\tau_{\gamma}=\tau_{\gamma}(p)$ in $\mathrm{GL}(G,\mathbb{F}_p)$ by 
\begin{eqnarray*}
\sigma_{\gamma}&=&e_{e_{G},\gamma}^1, \\
\tau_{\gamma}&=&\textrm{$($the permutation matrix by the shift on $G$ by the right multiplication of $\gamma$$)$}.
\end{eqnarray*}
More precisely, for $g,h\in G$, $(\tau_{\gamma})_{g,h}$ equals $1$ if $h=g\gamma$ and $0$ otherwise.
Note that if $\# (G)<\infty$ and if $\theta_{\gamma}\in \mathrm{Sym}(G)$ has a positive sign, then $\tau_{\gamma}$ belongs to $\mathrm{SL}(G,\mathbb{F}_p)$. 

Then, we have the following Cayley convergence, which may be seen as an analog of Lemma~\ref{lemma=SymmetricGroups} and Lemma~\ref{lemma=AlternatingGroups} for the case of special linear groups.

\begin{lemma}[Encoding into special linear groups]\label{lemma=SpecialLinearGroups}

Let $p$ be a prime number. Let $(\mathbf{G}_{m})_{m\in \mathbb{N}}=((G_m;s_1^{(m)},\ldots ,s_k^{(m)}))_{m}$ be a LEF approximation of an infinite group $\mathbf{G}_{\infty}=(G_{\infty};s_1^{(\infty)},\ldots ,s_k^{(\infty)})$. Assume that for every $m\in \mathbb{N}\cup\{\infty\}$ and for every $j\in [k]$, $s_j^{(m)}\ne e_{G_m}$ holds and $\theta_{s_j^{(m)}}\in \mathrm{Sym}(G_m)$ has a positive sign. For $m\in \mathbb{N}\cup\{\infty\}$, let $\Theta_m$ be the subgroup of $\mathrm{GL}(G,\mathbb{F}_p)$ generated by the $2k$-marking
\[
(\sigma_{s_1^{(m)}}(p), \sigma_{s_2^{(m)}}(p),\ldots ,\sigma_{s_k^{(m)}}(p),\tau_{s_1^{(m)}}(p),\tau_{s_2^{(m)}}(p),\ldots ,\tau_{s_k^{(m)}}(p)).
\]
Then, for $m\in\mathbb{N}$, $\Theta_m$ equals $\mathrm{SL}(G_m,\mathbb{F}_p)$. Furthermore, we have the following Cayley convergence in $\mathcal{G}(2k)$:
\begin{eqnarray*}
& &(\mathrm{SL}(G_m,\mathbb{F}_{p});\sigma_{s_1^{(m)}}, \ldots ,\sigma_{s_k^{(m)}},\tau_{s_1^{(m)}},\ldots ,\tau_{s_k^{(m)}}) \\
\quad \stackrel{\mathrm{Cay}}{\longrightarrow}\quad & &(\Theta_{\infty};\sigma_{s_1^{(\infty)}}, \ldots ,\sigma_{s_k^{({\infty})}},\tau_{s_1^{({\infty})}},\ldots ,\tau_{s_k^{({\infty})}}).
\end{eqnarray*}
\end{lemma}

\begin{remark}
The Cayley limit group $\Theta_{\infty}$ as in Lemma~\ref{lemma=SpecialLinearGroups} is a subgroup of $\mathrm{SL}(G_{\infty},\mathbb{F}_p)\rtimes G_{\infty}$. Here for an infinite countable group $G$, the action of $G$ in the group $\mathrm{SL}(G,\mathbb{F}_p)\rtimes G$ is given by permutations of coordinates induced by right multiplication $G\curvearrowleft G$. Moreover, the intersection of $\Theta_{\infty}$ and $G_{\infty}$ (on the right side of $\mathrm{SL}(G_{\infty},\mathbb{F}_p)\rtimes G_{\infty}$) equals $G_{\infty}$, Therefore, $\Theta_{\infty}$ is a locally-finite-lift of $G_{\infty}$.
\end{remark}

\begin{proof}[Proof of Lemma~$\ref{lemma=SpecialLinearGroups}$]
To see that for every $m\in \mathbb{N}$, $\Theta_m=\mathrm{SL}(G_m,\mathbb{F}_p)$ holds, first observe that every $s_j^{(m)}\in G_m$, $j\in [k]$, is a torsion (because $\#(G_m)<\infty$). Hence for every $j\in [k]$, $\sigma_{(s_j^{(m)})^{-1}}$ belongs to $\Theta_m$. Indeed, employ $\sigma_{(s_j^{(m)})}$ and powers of $\tau_{(s_j^{(m)})}$. It then follows that for every $\gamma\in G_m\setminus \{e_{G_m}\}$, $\sigma_{\gamma}$ may be written as some product of $\sigma_{s_1^{(m)}}, \ldots ,\sigma_{s_k^{(m)}},\tau_{s_1^{(m)}},\ldots ,\tau_{s_k^{(m)}}$; hence it is in $\Theta_m$. Since $\mathbb{F}_p$ is a field, it holds by Gaussian elimination that $\mathrm{SL}(G_m,\mathbb{F}_p)$ is generated by elements of the form above. Hence, $\Theta_m=\mathrm{SL}(G_m,\mathbb{F}_p)$.

To prove the latter Cayley convergence, we may argue in a similar manner to one in the proof of Lemma~\ref{lemma=SymmetricGroups}, with the local point of view.
\end{proof}

\subsection{Proof of Theorem~\ref{mtheorem=SpectralGap}}\label{subsection=SpectralGap}

Before our proof, we give the definition of \textit{elementary group} over a ring. Let $R$ be an associative ring with unit and $n\in \mathbb{N}_{\geq 2}$. Then the \textit{elementary group} of degree $n$ over $R$, here we write as $\mathrm{E}(n,R)$, is defined by the subgroup of $\mathrm{GL}(n,R)$ generated by elementary matrices $e_{i,j}^r$, $i\ne j\in [n]$, $r\in R$. A celebrated theorem by Ershov and Jaikin-Zapirain \cite{ErshovJaikinZapirain} states that for every $n\in \mathbb{N}_{\geq 3}$ and for every finitely generated (unital associative) $R$, the group $\mathrm{E}(n,R)$ has property $(\mathrm{T})$. In particular, it applies to the case where $R=\mathbb{Z}\langle X,Y\rangle$, which denotes the non-commutative polynomial ring over $\mathbb{Z}$ with indeterminates  $X$ and $Y$.

\begin{remark}\label{remark=FLq}
In \cite{MimuraUpgrading}, the author proved the following strengthening of the aforementioned theorem: For every $n\in \mathbb{N}_{\geq 4}$ and for every finitely generated (unital associative) $R$, the group $\mathrm{E}(n,R)$ has the fixed point property  with respect to $L_r$-spaces for \textit{all} $1<r<\infty$. Indeed, the fixed point property above is known to be strictly stronger than property $(\mathrm{T})$, which is equivalent for countable groups to the fixed point property with respect to $L_2$-spaces. See \cite{BaderFurmanGelanderMonod} for details. See also an expository article \cite{MimuraExpository} for the method of ``\textit{upgrading fixed points}'', which was employed in \cite{MimuraUpgrading} to prove the  fixed point property above. This reference \cite{MimuraExpository} also provides a simpler alternative proof (but without supplying any estimate of Kazhdan constants)
 of the theorem of Ershov and Jaikin-Zapirain. 
\end{remark}

\begin{proof}[Proof of $Theorem~\ref{mtheorem=SpectralGap}$]
First we recall the resolution of the Lubotzky--Weiss conjecture (property $(\mathrm{T})$ part) in \cite{ErshovJaikinZapirain} in the following form: \textit{For a sequence $(l_m)_{m\in \mathbb{N}}$ of strictly increasing sequence of positive integers divisible by $4$ and for a prime $p$, there exists a finitely generated dense subgroup $\Lambda_3$ of $K=\prod_{m\in \mathbb{N}}\mathrm{SL}(l_m,\mathbb{F}_p)$ with property $(\mathrm{T})$.} The proof was done in \cite[Subsection~6.3]{ErshovJaikinZapirain}; nevertheless, we include a (slightly different but essentially the same) proof for convenience of the reader. Recall that in our construction, we can arrange $(l_m)_{m\in \mathbb{N}}$ such that every $l_m$ is divisible by $4$; set $l_m'=l_m/4(\in \mathbb{N})$. Then by regarding $\mathrm{SL}(l_m,\mathbb{F}_p)=\mathrm{SL}(4l_m',\mathbb{F}_p)$ as an elementary group over $l_m'\times l_m'$ block matrices, we obtain a natural isomorphism
\[
\mathrm{SL}(l_m,\mathbb{F}_p)(=\mathrm{E}(l_m,\mathbb{F}_p))\ \simeq \ \mathrm{E}(4,\mathrm{Mat}_{l_m'\times l_m'}(\mathbb{F}_p)).
\]
Observe that the coefficient ring $\mathrm{Mat}_{l_m'\times l_m'}(\mathbb{F}_p)$ in the right hand side is \textit{always $2$-generated with ring unit $I_{l_m'}$ as a ring}: Indeed, if we set $x^{(m)}=e_{1,2}^1$ and $y^{(m)}$ as a matrix associated with a cyclic permutation on $[l_m']$, then $\{I_{l_m'},x^{(m)},y^{(m)}\}$ generates $\mathrm{Mat}_{l_m'\times l_m'}(\mathbb{F}_p)$. Now we consider $\mathrm{E}(4,\mathbb{Z}\langle X,Y\rangle)$ and fix a marking of it (for instance, we may take $T=(e_{1,2}^1,e_{1,2}^X,e_{1,2}^Y,\tau)$, where $\tau$ is a matrix associated with a cyclic permutation on $[4]$ with some minus-sign). Then for each $m\in \mathbb{N}$, the map sending $1$ to $I_{l_m'}$, $X$ to $x^{(m)}$ and $Y$ to $y^{(m)}$ induces a group quotient map
\[
\mathrm{E}(4,\mathbb{Z}\langle X,Y\rangle) \twoheadrightarrow \mathrm{E}(4,\mathrm{Mat}_{l_m'\times l_m'}(\mathbb{F}_p))\quad(\simeq \mathrm{SL}(l_m,\mathbb{F}_p)).
\]
This map projects the fixed marking $T$ of $\mathrm{E}(4,\mathbb{Z}\langle X,Y\rangle)$ to the corresponding marking $T_m$ of $\mathrm{SL}(l_m,\mathbb{F}_p)$. Finally, set $\Lambda_3$ as the underlying group of $\Delta_{m\in \mathbb{N}}((\mathrm{SL}(l_m,\mathbb{F}_p);T_m))$. It is dense in 
\[
K=\prod_{m\in \mathbb{N}}\mathrm{SL}(l_m,\mathbb{F}_p)
\]
by Lemma~\ref{lemma=Goursat}. Indeed, since $l_m\geq 4$, the only  possible finite simple quotient is $\mathrm{PSL}(l_m,\mathbb{F}_p)$. We claim that $\Lambda_3$ has property $(\mathrm{T})$. To prove this, observe that $\Lambda_3$ is a group quotient of $\mathrm{E}(4,\mathbb{Z}\langle X,Y\rangle)$ because each $(\mathrm{SL}(l_m,\mathbb{F}_p);T_m)$, $m\in \mathbb{N}$, satisfies every relation on $(\mathrm{E}(4,\mathbb{Z}\langle X,Y\rangle);T)$. We, in particular, obtain a $4$-generated example of $\Lambda_3$. (Note that this argument together with Remark~\ref{remark=FLq} shows that the $\Lambda_3$ above has the fixed point property for $L_r$-spaces for all $r\in (1,\infty)$.)

Then, we proceed to  a construction of $\Lambda_4\curvearrowright K$. Fix a prime $p$. Given $H$, we construct two sequences of $2$-marked groups $((L_m;w^{(m)},(u^{(m)})^3))_{m}$ and $((L_m;w^{(m)},u^{(m)}))_{m}$ obtained in the full proof of Theorem~\ref{mtheorem=MainTheorem}. Apply Lemma~\ref{lemma=SpecialLinearGroups} to these two sequences. Set $l_m=\# (L_m)$ and identify $\mathrm{SL}(l_m,\mathbb{F}_p)$ with $\mathrm{SL}(L_m,\mathbb{F}_p)$. We construct two systems of $4$-markings
\begin{eqnarray*}
& &((\mathrm{SL}(l_m,\mathbb{F}_{p}); \sigma_{w^{(m)}},\sigma_{(u^{(m)})^3},\tau_{w^{(m)}},\tau_{(u^{(m)})^3}))_{m\in \mathbb{N}},\\
& \textrm{and }&((\mathrm{SL}(l_m,\mathbb{F}_{p}); \sigma_{w^{(m)}},\sigma_{(u^{(m)})^3},\tau_{w^{(m)}},\tau_{u^{(m)}}))_{m\in \mathbb{N}}.
\end{eqnarray*}
Let $\Lambda_1$ (respectively, $\Lambda_2$) be the underlying group of the diagonal product of the former sequence (respectively, the latter sequence). Then, by Lemmata~\ref{lemma=SpecialLinearGroups}, \ref{lemma=DiagonalProducts} and \ref{lemma=Goursat}, we have the following:
\begin{itemize}
  \item $\Lambda_1\leqslant \Lambda_2$; 
  \item $\Lambda_1$ is a locally-finite-lift of $\mathbb{Z}$; 
  \item $\Lambda_2$ contains an isomorphic copy of $H$; and
  \item $\Lambda_1$ and $\Lambda_2$ are both dense in $K$.
\end{itemize}
Observe that by construction, we may arrange $(l_m)_{m\in \mathbb{N}}$ such that each $l_m$ is divisible by $4$. Let $\Lambda_3$ be another ($4$-generated) dense subgroup of $K$ that has property $(\mathrm{T})$, constructed in our first argument. Finally, let $\Lambda_4$ be the group generated by these $\Lambda_2$ and $\Lambda_3$ (which is $8$-generated and dense in $K$). We claim that the $\Lambda_1$ and $\Lambda_4$ above satisfy all of the three conditions as in Theorem~\ref{mtheorem=SpectralGap}. Indeed, the first and second conditions are fulfilled by construction. For the third condition on spectral gaps, since $\Lambda_3$ has property $(\mathrm{T})$, $\Lambda_3\curvearrowright K$ has a spectral gap. Now by density of $\Lambda_3$ in $K$, $\Lambda_3\curvearrowright K$ is  ergodic; recall our discussion in Remark~\ref{remark=Sawicki}. From this, it is clear that the spectral gap property for $\Lambda_4\curvearrowright K$ follows from that for $\Lambda_3\curvearrowright K$. See also Remark~\ref{remark=Expanders} for an alternative argument for the spectral gap property.
\end{proof}


\begin{remark}\label{remark=Expanders}
It is straightforward (see \cite[Proposition~1.15]{LubotzkyZuk} and \cite{AbertElek}) to see that for a finitely generated RF group $\Lambda$ and a chain $(N_m)_{m\in \mathbb{N}}$ of normal subgroups (recall Subsection~\ref{subsection=Profinite}), the profinite action $\varprojlim_{m} (\Lambda\curvearrowright \Lambda/N_m)$ has a spectral gap if and only if the actions $(\Lambda\curvearrowright \Lambda/N_m)_{m\in \mathbb{N}}$ has a \textit{uniform spectral gap}. It means for some (equivalently, every) finite generating set $S$ of $\Lambda$, the value of (best possible) $\epsilon_S$ as in the definition above of having spectral gaps for the action $\Lambda\curvearrowright \Lambda/N_m$, $m\in \mathbb{N}$, is uniformly bounded away from zero; this condition is equivalent to saying that the sequence of the Cayley graphs of $((\Lambda/N_m;S\ \mathrm{mod}\ N_m))_{m\in \mathbb{N}}$ forms an \textit{expander family} (\cite[5.6]{bookNowakYu}). It may be also restated as $\Lambda$ has \textit{property $(\tau)$} with respect to the chain $(N_m)_{m\in \mathbb{N}}$; see \cite{LubotzkyZuk} for details on property $(\tau)$.

In this point of view, we may have a more graph-theoretical proof of the spectral gap property for $\Lambda_4\curvearrowright K$ in the proof of Theorem~\ref{mtheorem=SpectralGap}, as follows: By argument in Subsection~\ref{subsection=Profinite}, it suffices to show that $(\Lambda_4\curvearrowright \prod_{m\in\mathbb{M}_n}\mathrm{SL}(l_m,\mathbb{F}_p))_n$ produces an expander family. Since $\Lambda_3$ has property $(\mathrm{T})$, the system $(\Lambda_3\curvearrowright \prod_{m\in\mathbb{M}_n}\mathrm{SL}(l_m,\mathbb{F}_p))_n$ produces an expander family. Now note that (as long as degrees of graphs are uniformly bounded), \textit{being an expander family is a monotone property}, that means, this property is closed under adding edges to each component of the original graph sequences: This is clear if we consider the characterization of expander families in terms of (Cheeger) isoperimetric constants (\cite[Definition 5.6.3]{bookNowakYu}), or of Poincar\'{e}-type inequalities. Therefore, $(\Lambda_4\curvearrowright \prod_{m\in\mathbb{M}_n}\mathrm{SL}(l_m,\mathbb{F}_p))_n$ yields an expander family, as desired.
\end{remark}

\begin{remark}\label{remark=Lift}
It is \textit{un}clear  whether the action $\Lambda\curvearrowright \prod_{m\in \mathbb{N}}L_m$ has a spectral gap, even if the sequence of the Cayley graphs of $((L_m;v_1^{(m)},\ldots ,v_{\ell}^{(m)}))_m$ forms an expander family. Here $(\mathbf{L}_m)_{m\in \mathbb{N}}=((L_m;v_1^{(m)},\ldots ,v_{\ell}^{(m)}))_m$ is a LEF approximation that satisfies the condition of (the former statement of) Lemma~\ref{lemma=Goursat}, and $\Lambda$ is the underlying group of $\Delta_{m\in \mathbb{N}}(\mathbf{L}_m)$. This is because, to switch to the corresponding profinite system, we need to lift $\mathbf{L}_n$ to $\Delta_{m\in \mathbb{M}_n}(\mathbf{L}_m)$; recall Remark~\ref{remark=Expanders}. Even if the sequence of the Cayley graphs of $(\mathbf{L}_n)_{n\in \mathbb{N}}$ forms an expander family, it may not be clear  whether the same holds for $(\Delta_{m\in \mathbb{M}_n}(\mathbf{L}_m))_{n\in \mathbb{N}}$. This problem is a special case of \cite[Question~1.14]{LubotzkyZuk}; in \cite[Corollary~9]{AbertElek}, a counterexample to the original question of \cite{LubotzkyZuk} was constructed, but it is a family of non-normal finite index subgroups.

\end{remark}

\begin{remark}\label{remark=AlternatingGroups}
A remarkable result of Kaluba--Nowak--Ozawa \cite{KalubaNowakOzawa}, property $(\mathrm{T})$ for $\mathrm{Aut}(F_5)$, together with a result of Gilman implies the following: There exists a strictly increasing sequence $(c_m)_{m\in \mathbb{N}}$ of integers at least $5$ such that the group $\prod_{m\in \mathbb{N}}\mathrm{Alt}([c_m])$ admits a finitely generated dense subgroup with property $(\mathrm{T})$. 

We may obtain $\Lambda_4$ as in the statement of Theorem~\ref{mtheorem=SpectralGap} for the case where $K_{\mathrm{Alt}}=\prod_{m\in \mathbb{N}}\mathrm{Alt}([l_m])$, where $(l_m)_{m\in \mathbb{N}}$ is some sequence of strictly increasing natural numbers at least $5$ (but numbers of generators of $\Lambda_4$ may get enormous). Indeed, to construct such a $\Lambda_4$, we appeal to the following result: Kassabov \cite{KassabovSymmetric}, together with Kassabov--Nikolov \cite{KassabovNikolov}, constructed a finitely generated dense subgroup $\Lambda_3$ of $\prod_{m\in \mathbb{N}_{\geq 5}}\mathrm{Alt}([m])$ such that it has property $(\tau)$ (with respect to the family of all finite index subgroups); see also \cite{Kassabov}. Compare with Remark~\ref{remark=Expanders}. Finally, construct $\Lambda_1$ and $\Lambda_2$ not by Lemma~\ref{lemma=SpecialLinearGroups} but by Lemma~\ref{lemma=AlternatingGroups}; obtain $\Lambda_1$ and $\Lambda_4=\langle \Lambda_2,\Lambda_3\rangle$, both dense in $K_{\mathrm{Alt}}$, as desired.
\end{remark}

\section*{Acknowledgments}
This work has been done during the  two-year stay of the author in the \'{E}cole Polytechnique F\'{e}d\'{e}rale de Lausanne supported by Grant-in-Aid for JSPS Oversea Research Fellowships, as well as his visit to University of Bristol during that stay. The author wishes to express his gratitude to Professor Nicolas Monod and Mrs. Marcia Gouffon at the EPFL and Professor John Mackay at the University of Bristol for their hospitality and help, respectively, on his stay and visit. The author thanks Yash Lodha and Michael Magee for discussion, respectively, on Hall's embedding theorem and on Goursat's lemma. He is grateful to Goulnara Arzhantseva, Yoshikata Kida and Sven Raum for several comments and suggestions, Piotr W. Nowak for the reference \cite{Nowak}, Damian Osajda for discussion on his construction in \cite{OsajdaRF}, Damian Sawicki for comments and the reference \cite{SawickiPropertyA}, and Ryokichi Tanaka for discussions on LFNF-lifts.

\bibliographystyle{plain}
\bibliography{mimura_nonexact.bib}

\end{document}